\documentclass[aap]{imsart}

\RequirePackage[OT1]{fontenc}
\RequirePackage{amsthm,amsmath,amssymb}
\RequirePackage[numbers]{natbib}
\RequirePackage[colorlinks,citecolor=blue,urlcolor=blue]{hyperref}
\RequirePackage{bbm}
\RequirePackage{bm}
\RequirePackage{enumerate}

\startlocaldefs
\newtheorem{theorem}{Theorem}[section]

\newtheorem{proposition}[theorem]{Proposition}

\newtheorem{lemma}[theorem]{Lemma}
\theoremstyle{plain}

\newcommand{\R}{\mathbb{R}}
\newcommand{\N}{\mathbb{N}}
\renewcommand{\P}{\mathbb{P}}

\newcommand{\E}{\mathbb{E}}
\newcommand{\1}{\mathbbm{1}}

\renewcommand{\S}{\mathcal{S}}
\newcommand{\C}{\mathcal{C}}

\newcommand{\ind}{\mathbbm{1}}

\newcommand{\Ps}{\mathbb{P}^{\omega,\tau}}
\newcommand{\Es}{\mathbb{E}^{\omega,\tau}}
\newcommand{\Pt}{\tilde{\mathbb{P}}}
\newcommand{\Et}{\tilde{\mathbb{E}}}
\newcommand{\A}{\mathcal{A}}
\newcommand{\Js}{{\mathcal{J}(\omega,\tau)}}

\newcommand{\eps}{\varepsilon}
  
\endlocaldefs

\arxiv{arXiv:1610.06000}
\begin{document}

\begin{frontmatter}
\title{Exceptional times of the critical dynamical Erd\H{o}s-R\'enyi graph}
\runtitle{Exceptional times}

\begin{aug}
  \author{\fnms{Matthew, I.} \snm{Roberts}\thanksref{m1}\ead[label=e1]{mattiroberts@gmail.com}}
\and
\author{\fnms{Bat\i{}} \snm{\c{S}eng\"ul}\thanksref{m2}\ead[label=e2]{batisengul@gmail.com}}

\runauthor{M. I. Roberts and B. \c{S}eng\"ul}

\affiliation{University of Bath\thanksmark{m1} and Bank of America, Merrill Lynch\thanksmark{m2}}

\address{University of Bath\\
Claverton Down\\
Bath\\
BA2 7AY\\
United Kingdom\\
\phantom{E-mail:\ }\printead*{e1}}

\address{2 King Edward St\\
  London\\
  EC1A 1HQ\\
  United Kingdom\\
\phantom{E-mail:\ }\printead*{e2}}
\end{aug}

\begin{abstract}
  In this paper we introduce a network model which evolves in time, and study its largest connected component.
  We consider a process of graphs $(G_t:t \in [0,1])$, where initially we start with a critical Erd\H{o}s-R\'enyi graph ER$(n,1/n)$, and then evolve forwards in time by resampling each edge independently at rate $1$.
  We show that the size of the largest connected component that appears during the time interval $[0,1]$ is of order $n^{2/3}\log^{1/3} n$ with high probability. This is in contrast to the largest component in the static critical Erd\H{o}s-R\'enyi graph, which is of order $n^{2/3}$.
\end{abstract}



\end{frontmatter}

\section{Introduction and main result}\label{sec:intro}

An Erd\H{o}s-R\'enyi graph ER$(n,p)$ is a random graph on $n$ vertices $\{1,\dots,n\}$, where each pair of vertices is connected by an edge with probability $p$, independently of all other pairs of vertices.
Erd{\H{o}}s and R{\'e}nyi~\cite{erdos_renyi_on} introduced this graph (or rather a very closely related graph) and examined the structure of its connected components.
Since then, Erd\H{o}s-R\'enyi graphs have been intensively studied and have become a cornerstone of probability and combinatorics: see for example~\cite{bollobas_book, durrett_book, janson_book}~and references therein.

Let $L_n$ denote the largest connected component of an Erd\H{o}s-R\'enyi graph ER$(n,p)$ with $p=\mu/n$. We write $|L_n|$ for the number of vertices in $L_n$. 
This quantity exhibits a phase transition as $\mu$ passes $1$:
\begin{list}{(\roman{enumi})}{\usecounter{enumi}}
\item if $\mu <1$, then $(\log n)^{-1}|L_n|$ converges in probability to $1/\alpha(\mu)$ where $\alpha(\mu)=\mu-1-\log\mu\in(0,\infty)$ (see \cite[Corollaries 5.8 and 5.11]{bollobas_book});
\item if $\mu =1$, then $n^{-2/3}|L_n|$ converges in distribution to some non-trivial random variable as $n \rightarrow \infty$ (see \cite{aldous_mult_coal});
\item if $\mu >1$, then $n^{-1}|L_n|$ converges in probability to $\theta(\mu)$ where $\theta(\mu)\in (0,1)$ is the unique solution to $\theta(\mu) = 1-e^{-\mu \theta(\mu)}$ (see \cite[Theorem 5.4]{janson_book}).
\end{list}
The model ER$(n,1/n)$ is therefore referred to as the critical Erd\H{o}s-R\'enyi graph.

In this paper we study a dynamical version of the critical Erd\H{o}s-R\'enyi graph, a process of random graphs $(G_t:t \in [0,1])$ on the vertex set $\{1,\dots,n\}$, constructed as follows.
Initially $G_0$ is distributed as ER$(n,1/n)$.
Then the presence of each edge $vw$ between vertices $v\neq w$ is resampled at rate $1$, independently of all other edges. That is, at the times of a rate 1 Poisson process, we remove the edge $vw$ if it exists, and then place an edge with probability $1/n$, independently of everything else.
Clearly ER$(n,1/n)$ is invariant for this process, so for each $t \geq 0$, $G_t$ is a realisation of ER$(n,1/n)$. 
Let $L_n(t)$ denote the largest connected component of $G_t$.
Then for each fixed $t\in [0,1]$, $|L_n(t)|$ is of order $n^{2/3}$ with high probability as $n\to\infty$. Our main result gives a contrasting statement about the size of $\sup_{t\in[0,1]} |L_n(t)|$, showing that with high probability there are (rare) times when $|L_n(t)|$ is of order $n^{2/3}\log^{1/3}n$ (where we write $\log^\alpha n$ to mean $(\log n)^\alpha$).

\begin{theorem}\label{thm:noise}
  As $n\to\infty$,
  \[
    \P \left( \frac{\sup_{t \in [0,1]}|L_n(t)|}{n^{2/3}\log^{1/3}n}>\beta \right)\longrightarrow \begin{cases} 1 &\hbox{if } \beta < 2/3^{2/3}\\ 0 &\hbox{if } \beta\ge 2/3^{1/3}.\end{cases}
  \]
\end{theorem}

\noindent
We will also give a result on the noise sensitivity of component sizes in Proposition \ref{prop:ns}, once we have developed the required notation.

\subsection{Further discussion around Theorem \ref{thm:noise}}\label{sec:discussion}

It is not difficult to deduce from known results (see for example \cite{bollobas_book}) together with a first moment method (see Section \ref{sec:large_beta}) that for Erd\H{o}s-R\'enyi graphs away from criticality, the size of the biggest component in the dynamical model is of the same order as in the static model. That is, for ER$(n,\mu/n)$ with $\mu < 1$, $\sup_{t\in[0,1]}|L_n(t)|$ is of order $\log n$ with high probability; and for ER$(n,\mu/n)$ with $\mu>1$, $\sup_{t\in[0,1]}|L_n(t)|$ is of order $n$ with high probability. The critical graph ER$(n,1/n)$ is therefore the most interesting case.

Returning to ER$(n,1/n)$, the obvious open questions posed by Theorem \ref{thm:noise} are:
\begin{itemize}
  \item Does $\sup_{t\in[0,1]}|L_n(t)|/(n^{2/3}\log^{1/3} n)$ converge in probability as $n\to\infty$? If so, what is its limit?
  \item What does the set of exceptional times, i.e.~$\{t\in[0,1] : |L_n(t)|\ge \beta n^{2/3}\log^{1/3} n\}$, look like?
  \item What does the largest component look like at exceptional times?
  \item How does $\inf_{t\in[0,1]} |L_n(t)|$ behave?
  \item What if we resample each edge at rate $n^{\gamma}$ for $\gamma\neq 0$?
\end{itemize}

For the first question we conjecture that $\sup_{t\in[0,1]}|L_n(t)|/(n^{2/3}\log^{1/3} n)\to 2/3^{1/3}$ in probability as $n\to\infty$. We hope to address this in future work, but substantial further technical estimates are required.

We can say a limited amount about the second question. On the one hand, it is easy to check that the Lebesgue measure of the set of times at which there is a component of size at least $A_n n^{2/3}$ converges in probability to zero whenever $A_n\to\infty$, so certainly the Lebesgue measure of the set of exceptional times converges in probability to zero for any $\beta>0$. On the other hand, let $X_\beta(t)$ be $1$ when the largest component is larger than $\beta n^{2/3}\log^{1/3}n$, and $0$ at other times. For $\delta>0$, let $N_\beta(\delta)$ be the number of times in the interval $[0,\delta]$ at which $X_\beta(t)$ changes its value. Jonasson and Steif \cite[Corollary 1.6]{steif_jonasson_volatility} showed that if $\P(|L_n|\ge \beta n^{2/3}\log^{1/3} n)\to 0$ but $\P(N_\beta(1)\ge 1)\to 1$, then $N_\beta(\delta)\to\infty$ in distribution as $n\to\infty$ for any fixed $\delta>0$. Theorem \ref{thm:noise} tells us that these conditions hold for $\beta<2/3^{2/3}$.

Going further than this, one might like to know whether the set of exceptional times $\{t\in [0,1]: |L_n(t)|\geq \beta n^{2/3}\log^{1/3} n\}$ converges in distribution as $n\to\infty$, and if so, what the Hausdorff dimension of this limiting set is. We conjecture that the Hausdorff dimension is $(1-3\beta^3/8)\vee 0$. This conjecture follows naturally from a simple box counting argument using the sets $\mathcal E_i$ from our upper bound in Section \ref{sec:large_beta}. Again, we hope to investigate this in future work.

For the third question, it is natural to guess---in analogy with work on dynamical planar lattice percolation by Hammond, Pete and Schramm \cite{MR3433575}---that the largest component at ``typical'' exceptional times looks like a static component conditioned to have size at least $\beta n^{2/3}\log^{1/3} n$. Unfortunately our combinatorial method for estimating the probability that such a component exists (using results from \cite{roberts:ER}) gives little insight into its structure. Analysis using Brownian excursions, after Aldous \cite{aldous_mult_coal} and Addario-Berry, Broutin and Goldschmidt \cite{broutin_et_al:continuum_limit_critical_rgs}, might shed more light on this problem.

The fourth question appears to be substantially different from Theorem \ref{thm:noise} and would require a different approach.

For the fifth question, the most interesting case is $\gamma=-1/3$. If we rescale component sizes by $n^{2/3}$ then, based on Aldous' multiplicative coalescent \cite{aldous_mult_coal}, we expect to see something like a multiplicative fragmentation-coalescent process. Rossignol \cite{rossignol} has shown that this is indeed the case.

\subsection{Background}\label{sec:bg}
Dynamical percolation was introduced by H{\"a}ggstr{\"o}m, Peres and Steif~\cite{MR1465800}. Take a graph $G=(V,E)$ and create a dynamical random graph $(G_t,\, t\ge0)$ as follows. Each edge $e\in E$ is present at time $0$ with probability $p\in[0,1]$, independently of all others. Each edge is then rerandomized independently at the times of a rate 1 Poisson process. This model is known as \emph{dynamical bond percolation} on $G$ with parameter $p$. (Alternatively we may say that each \emph{vertex} is present with probability $p$ and rerandomized at rate 1; this is known as dynamical site percolation.) The model that we investigate in this paper is then simply dynamical bond percolation when $G$ is the complete graph on $n$ vertices and $p=1/n$.

A question of particular interest for infinite graphs is whether there exists a time at which an infinite component appears. Schramm and Steif~\cite{MR2630053} were able to show that for critical ($p=1/2$) dynamical site percolation on the triangular lattice, almost surely, there are times in $[0,\infty)$ when an infinite component is present, even though at any fixed time $t$ there is almost surely no infinite component. The times at which an infinite component exists are then known as \emph{exceptional times}. Their proof relied on tools from Fourier analysis, randomized algorithms, and the theory of noise sensitivity of Boolean functions as introduced by Benjamini, Kalai and Schramm~\cite{MR1813223}. We will see similar methods appearing in our proof, although in each case there will be a non-standard approach required.

Since its introduction roughly 20 years ago, dynamical percolation has been studied intensively~\cite{MR1959784,MR2223951,MR2736153,MR1716767,MR3433575,MR2521411} in various settings (see also~\cite{MR2883377,garban_steif_noise,MR2762676} and references within). Most of the study has so far been restricted to infinite graphs and the question of existence of an infinite component. Of the very few results on finite graphs, Lubetzky and Steif \cite{MR3433581} studied the noise sensitivity properties of various Boolean functions related to Erd\H{o}s-R\'enyi graphs; and Jonasson and Steif \cite{steif_jonasson_volatility} gained results about dynamical percolation on infinite spherically symmetric trees restricted to the first $n$ levels, in the context of what they call the volatility of Boolean functions (which we mentioned using different notation in Section \ref{sec:discussion} in the discussion around the second open question).
  
The critical Erd\H{o}s-R\'enyi graph is one of the simplest models of random networks. Several more complex random graph models, such as preferential attachment graphs, have since been introduced in an attempt to more realistically model the features seen in real-world networks such as the world-wide web; see \cite{vanrandom} for an overview. The model we consider in this paper is known within the network science literature as a \emph{temporal network}. The report~\cite{Holme201297} gives a good introduction to the subject. There is interest in comparing real networks with random models, and in order to do this for networks that change with time (temporal networks), \cite{Holme201297} details several ways of constructing dynamical random graphs loosely based around the configuration model. Our dynamical Erd\H{o}s-R\'enyi model is simpler, and we hope that it will lead to further progress in the probabilistic community in investigating other temporal network models.

  \subsection{Proof ideas for Theorem \ref{thm:noise}}
  The proof that if $\beta> 2/3^{1/3}$ then there are no exceptional times (i.e.~with high probability there are no times in $[0,1]$ when there is a component of size bigger than $\beta n^{2/3}\log^{1/3} n$) uses a standard first moment method. We split $[0,1]$ into many smaller intervals, use known asymptotics for the probability of seeing a large component for $p$ slightly bigger than $1/n$ to bound the probability of seeing an exceptional time on one of these small intervals, and then take a union bound. The main interest of this paper is therefore the result that there are exceptional times for $\beta< 2/3^{2/3}$.

  As discussed in \cite{MR2630053}, in order to see such times, the configuration must ``change rapidly'' so that it has ``many chances'' to have a large component. By ``change rapidly'' we mean that the configurations must have small correlations over short time intervals. To quantify this we use a second moment method, and the key will be to estimate
  \[\int_0^1\int_0^1 \P(|\C_u(s)|>An^{2/3}, \,|\C_v(t)|>An^{2/3}) \,{\rm d}t\,{\rm d}s\]
  where $A = \beta \log^{1/3} n$ and $\C_v(t)$ is the connected component containing vertex $v$ at time $t$.

  We will need different methods for estimating $\P(\C_u(0)|>An^{2/3}, \,|\C_v(t)|>An^{2/3})$, roughly depending on whether $|t-s|$ is less than or greater than $n^{-2/9}$. For small values of $|t-s|$ we will use a counting argument. For larger values of $|t-s|$ the correlations become harder to control and we will need to use tools from discrete Fourier analysis. A very interesting theory of noise sensitivity has been developed around this concept when $\P$ is a uniform product measure, i.e.~when the probability that each edge (or vertex) is present is $1/2$: see~\cite{garban_steif_noise}. Since our measure $\P$ is highly non-symmetric, we must redevelop some of the noise sensitivity tools in our non-standard setting. Even then there are complications and some twists on the theory are needed, which may be of interest in their own right.

  The basic idea is to use the notion of randomized algorithms. We aim to design an algorithm which examines some of the edges $e\in E$ (i.e.~looks at whether they are present or not), and decides whether or not there is a large component. If for any fixed $e$, the probability that the algorithm checks $e$ is small, then the Fourier coefficients that we are interested in must also be small. This result is known in the uniform case \cite{MR2630053}, and the proof carries over to non-uniform $\P$.
  The major complication here is that we are not able to construct an algorithm with the desired properties, essentially because of the lack of geometry in the graph. To check whether a particular vertex $v$ is in a large component, we need to examine almost all the edges emanating from $v$. For each $v$, we are therefore forced to consider two classes of edges. For those edges $e$ that do not have an endpoint at $v$, we can use a well-known exploration algorithm and use the lack of geometry to our advantage to bound the probability that $e$ is examined. For the edges that do have an endpoint at $v$, we use a completely different method inspired by the \emph{spectral sample} introduced in \cite{MR1813223}. We bound the relevant Fourier coefficients by looking at the probability that the edge $e$ is pivotal, i.e.~that there is a large component when $e$ is present and not if $e$ is absent.

  \section{Fourier analysis of Boolean functions}\label{sec:fourier}
  In this section we give general results on the Fourier analysis of Boolean functions. Several of the results presented here are known in the case when $\P$ is a uniform product measure; see for example~\cite{garban_steif_noise}. We also note that Talagrand~\cite{MR1303654} developed hypercontractivity results in the case we are considering, where $\P$ is a homogeneous but non-uniform product measure. We repeat some of his definitions below.

  \subsection{Definitions and first results}\label{sec:defns}

  Let $E$ be a finite set and define $\Omega := \{0,1\}^E$. Let $\P=\P_p$ be a measure on $\Omega$ defined by
  \[
    \P(\omega)=p^{\#\{e:\omega(e)=1\}}\left(1-p\right)^{\#\{i:\omega(e)=0\}}.
  \]
  All of our results in this section apply to any finite set $E$ and any $p\in[0,1]$. We refer to the elements of $E$ as bits. Of course we have in our minds the application where $E$ is the edge set of the complete graph $K_n$ and $p=1/n$;  and where for $\omega \in \Omega$ we say that edge $e$ is present if and only if $\omega(e)=1$.

  For $\omega \in \Omega$ and $e \in E$ let
  \[
    r_e(\omega):=\begin{cases}
      \sqrt{\frac{1-p}{p}} & \text{ if } \omega(e)=1\\
      -\sqrt{\frac{p}{1-p}} & \text{ if } \omega(e)=0.
    \end{cases}
  \]
  For $S \subset E$ let
  \[
    \chi_S(\omega)= \prod_{e \in S} r_e(\omega)
  \]
  where we set $\chi_\emptyset\equiv 1$. Then for any function $f:\Omega\to\R$ and $S\subset E$ we define
  \[
    \hat f(S)=\E[f\chi_S],
  \]
  and call $\hat f(S)$, for $S\subset E$, the \emph{Fourier coefficients} of $f$.

  It is easy to check that $\{\chi_S: S\subset E\}$ forms an orthonormal basis for $L^2(\P)$, and therefore---just as in continuous Fourier analysis---the function $\hat f$ encodes all of the information about $f$, in that $f(\omega)=\sum_{S\subset E} \hat f(S)\chi_S(\omega)$.

  One simple consequence of the definition is that $\E[f] = \hat f(\emptyset)$.
  Another useful result is \emph{Plancherel's identity}, which states that for two functions $f,g:\Omega\rightarrow \R$,
  \begin{equation}\label{eq:plancherel}
    \sum_{S} \hat f(S)\hat g(S) = \E[f g].
  \end{equation}
  This is easy to prove simply by writing out $f=\sum_{S}\hat f(S) \chi_S$ and $g=\sum_{S'}\hat g(S')\chi_{S'}$ and using orthonormality.

  Recall from the introduction that we will be interested in bounding probabilities like $\P(|\C_u(s)|>An^{2/3}, \,|\C_v(t)|>An^{2/3})$, where $\C_v(t)$ is the component containing $v$ at time $t$ in the dynamical Erd\H os-R\'enyi graph. We will therefore be appling our Fourier analysis to functions of the form $\ind_{\{\C_v(t)>An^{2/3}\}}$. The following lemma, which is already known (see~\cite[(4.2)]{garban_steif_noise} and~\cite[(2.2)]{MR1813223}), will be very useful for this purpose. Given $\omega\in\Omega$ and $\eps\in[0,1]$, let $\omega_\eps$ be the configuration obtained by rerandomizing each of the bits in $\omega$ independently with probability $\eps$. That is, for each $e\in E$,
  \[\omega_\eps(e) = \omega(e)\ind_{\{U_e>\eps\}} + \ind_{\{V_e<p\}}\ind_{\{U_e\le \eps\}}\]
  where $U_e$ and $V_e$ are independent uniform random variables on $(0,1)$.

  \begin{lemma}\label{lemma:noise_fourier}
    For any $\eps\in[0,1]$ and any $f,g:\Omega\to\R$,
    \[\E[f(\omega)g(\omega_\eps)]=\sum_{S} \hat f(S)\hat g(S) (1-\eps)^{|S|}\]
    (where the expectation $\E$ averages both over $\omega\in \Omega$ and also over the randomness in the resampling required to create $\omega_\eps$).
  \end{lemma}

  \begin{proof}
    Note that
    \begin{align*}
    \E[f(\omega)g(\omega_\eps)] &= \E\Big[\sum_S \hat f(S)\chi_S(\omega) \sum_{S'}\hat g(S')\chi_{S'}(\omega_\eps)\Big] \\
                                &= \sum_{S,S'}\hat f(S)\hat g(S') \E[\chi_S(\omega)\chi_{S'}(\omega_\eps)].
  \end{align*}
    It is easy to check that if $S\neq S'$ then $\E[\chi_S(\omega)\chi_{S'}(\omega_\eps)]=0$, and on the other hand that
    \[\E[\chi_S(\omega)\chi_S(\omega_\eps)] = \prod_{e\in S} \E[r_e(\omega)r_e(\omega_\eps)] = (1-\eps)^{|S|}.\qedhere\]
  \end{proof}
  
  \subsection{Noise sensitivity}
  
  At this point we veer from our main path for a while to state a result about the noise sensitivity of component sizes in Erd\H{o}s-R\'enyi graphs. Following the notation from Section \ref{sec:defns}, suppose that we have a sequence of functions $F_n:\Omega_n\to\R$, $n\ge 1$, where $\Omega_n = \{0,1\}^{E_n}$. Recall that for $\omega\in\Omega_n$ and $\eps\in[0,1]$, we let $\omega_\eps$ be the configuration obtained by rerandomizing each of the bits in $\omega$ independently with probability $\eps$.
  
  We say that the sequence $(F_n)_{n\ge 1}$ is \emph{noise sensitive} if
  \[\E[F_n(\omega)F_n(\omega_\eps)] - \E[F_n(\omega)]^2 \to 0 \hspace{4mm} \hbox{as $n\to\infty$.}\]
  For a sequence $(\eps_n)_{n\ge1}$, we say that $F_n$ is \emph{quantitatively noise sensitive with scaling $\eps_n$} if
  \[\E[F_n(\omega)F_n(\omega_{\eps_n})] - \E[F_n(\omega)]^2 \to 0 \hspace{4mm} \hbox{as $n\to\infty$.}\]
  
   \begin{proposition}\label{prop:ns}
     For any fixed $a\in(0,\infty)$ and any $(\eps_n)_{n\ge 1}$ such that $\lim_{n\to \infty} n^{1/6}\eps_n = \infty$, the sequence of functions $F_n = \ind_{\{|L_n|\ge an^{2/3}\}}$ is quantitatively noise sensitive with scaling $\eps_n$.
   \end{proposition}
  
   Of course if $F_n\to 0$ with high probability (or if $F_n\to 1$ with high probability) then $F_n$ is trivially noise sensitive. It is well known---see \cite{aldous_mult_coal}---that $F_n$ as in Proposition \ref{prop:ns} is non-trivial in that sense, i.e.~$\lim_{n\to\infty}\P(F_n=0)\in(0,1)$.
   
  We believe that in fact the given sequence $F_n$ is quantitatively noise sensitive with scaling $\eps_n$ whenever $n^{1/3}\eps_n\to\infty$, and that this is the best possible such scaling. A simple argument, similar to the union bound in Section \ref{sec:large_beta}, shows that the functions $F_n(\omega)$ and $F_n(\omega_{\eps_n})$ defined in Proposition \ref{prop:ns} do not decorrelate when $\eps_n$ is of order $n^{-1/3}$. More precisely, in $\omega$, the total number of vertices in components of size at least $an^{2/3}/2$ is $O(n^{2/3})$ with high probability, so there are $O(n^{4/3})$ edges that would create a new component of size $n^{2/3}$ simply by being switched on, and the probability that any of these edges is switched on under $\omega_{n^{-1/3}}$ is bounded away from $1$. Similarly, the total number of edges that would break a component of size at least $an^{2/3}$ into two pieces of size $<an^{2/3}$ if switched off is $O(n^{1/3})$ with high probability, and the probability that any of these edges is switched off under $\omega_{n^{-1/3}}$ is again bounded away from $1$. This argument provides a lower bound for the noise sensitivity threshold, but it also makes a persuasive case for $n^{-1/3}$ being the correct threshold, since rerandomising significantly more edges would change the status of a large number of \emph{pivotal} edges (see Section \ref{sec:piv}).
  
  Lubetzky and Steif \cite{MR3433581} showed that $n^{-1/3}$ is the correct noise sensitivity threshold when $F_n$ is instead the indicator that the critical Erd\H{o}s-R\'enyi graph contains a cycle whose size is of order $n^{1/3}$. Roughly speaking, a cycle of order $n^{1/3}$ entails a component of order $n^{2/3}$, though it is possible to have components of order $n^{2/3}$ without having cycles of order $n^{1/3}$. Further, Rossignol \cite{rossignol} shows that the system of large components has a well-behaved scaling limit when edges are resampled at rate $n^{-1/3}$, which again suggests that faster rerandomization would break the correlation structure. Finally, this sensitivity threshold would also coincide with our conjecture for the existence of exceptional times for any $\beta<2/3^{1/3}$.
  
  A proof of Proposition \ref{prop:ns} will follow almost as a byproduct of our proof of Theorem \ref{thm:noise}. We carry out the details in Section \ref{sec:ns}.

  \subsection{Randomized algorithms and revealment}\label{sec:randalg}

  Evaluating Fourier coefficients directly is often quite difficult and instead we concentrate on bounding sums such as the one on the right-hand side of Lemma \ref{lemma:noise_fourier}. One approach that has proven fruitful in the past is to introduce a \emph{randomized revealment algorithm} that attempts to decide the value of the function $f$ by revealing $\omega(e)$ only for relatively few of the possible bits $e\in E$. If for any fixed $e$, the probability that the algorithm reveals $\omega(e)$ is small, then it turns out that the sum of the Fourier coefficients must be small \cite[Theorem 1.8]{MR2630053}. Our main result in this section is a generalization of \cite[Theorem 1.8]{MR2630053}.

  Let $f:\Omega=\{0,1\}^E \rightarrow \R$. A revealment algorithm, $A$, for $f$ is a sequence of bits $e_1,e_2,\ldots, e_T\in E$, chosen one by one, with the choice of $e_k$ possibly depending on the values of $\omega(e_1),\ldots,\omega(e_{k-1})$, and such that knowledge of $\omega(e_1),\ldots,\omega(e_T)$ determines the value of $f(w)$. A \emph{randomized revealment algorithm} is a revealment algorithm that is also allowed to use auxiliary randomness in making choices. Given such an algorithm $A$, let $J$ be the set of bits revealed by $A$.

  For $U\subset E$, define the revealment of the algorithm $A$ on $U$ by
  \[
    \mathcal R_U= \mathcal R_U(f,A) :=\max_{e \in U^c}\P(e \in J).
  \]
  Our main result in this section is the following generalization of \cite[Theorem 1.8]{MR2630053}.

  \begin{theorem}\label{thm:randalg}
    Let $A$ be an algorithm determining $f:\Omega\to\R$ and let $U\subset E$. Then for any $k \in \N$,
    \[
      \sum_{\substack{|S|=k,\\ S\cap U=\emptyset}}\hat f(S)^2 \leq \mathcal R_U(f,A) \E[f(\omega)^2] k.
    \]
  \end{theorem}

  The result in \cite{MR2630053}, besides being stated for the uniform measure (i.e.~$p=1/2$), only included the case $U=\emptyset$. The reason that we need a generalization involves the geometry of the Erd\H{o}s-R\'enyi graph. As far as we can tell, any algorithm to check whether there is an unusually large component must reveal almost all of the edges emanating from many of the vertices; similarly, any algorithm to check whether a particular vertex $v$ is in an unusually large component must reveal almost all of the edges with an endpoint at $v$.

  To get around this problem we fix a vertex $v$ and separate subsets $S$ of edges into those which contain an edge with an endpoint at $v$, and those which do not. We then use Theorem \ref{thm:randalg} to bound the Fourier coefficients of the latter sets, and take a different approach to the former. This different approach was inspired by the spectral sample introduced in \cite{MR1813223}, and will be carried out in Section \ref{sec:piv}.
  
  Schramm and Steif \cite{MR2630053} noted that it may be possible to improve their Theorem 1.8 for large $k$, and we believe similarly that our Theorem \ref{thm:randalg} may not be optimal. They suggest that the sum over $|S|=k$ might be changed to a sum over $|S|\le k$ with no change on the right-hand side, and such an improvement would allow us to give improved versions of Theorem \ref{thm:noise} and Proposition \ref{prop:ns} that are essentially best possible: convergence in probability of $\sup_{t\in[0,1]}|L_n(t)|/(n^{2/3}\log^{1/3}n)$ to $2/3^{1/3}$, and quantitative noise sensitivity for any $\eps_n\gg n^{-1/3}$.

  For now we aim to prove Theorem \ref{thm:randalg}. Our strategy is very much based on the proof in \cite{MR2630053}.

  Let $\tau\in\mathcal T$ represent the auxiliary randomness used by the algorithm, and let $\Pt$ be the canonical probability measure on the extended space $\Omega\times\mathcal T$. Let $\A$ be the smallest $\sigma$-algebra such that $J$ and $\{\omega(e) : e\in J\}$ are measurable. Note that since $A$ determines the value of $f$, and $\A$ contains all the information revealed by $A$, $f$ is $\A$-measurable.

  For a configuration $\omega'\in\Omega$ define the configuration $\omega'_{J(\omega,\tau)}$ by setting
\[\omega'_{J(\omega,\tau)}(e):=\begin{cases}
      \omega(e) & \text{ if } e \in J(\omega,\tau)\\
      \omega'(e) & \text{ if }e \notin J(\omega,\tau).
  \end{cases}\]
  Next, for any function $h:\Omega \rightarrow \R$ and $(\omega,\tau)\in \Omega\times\mathcal T$, define $h_{\Js}$ by
  \[\begin{aligned} h_\Js: \Omega &\to \R\\
                             \omega' &\mapsto h(\omega'_{\Js}).\end{aligned}\]
    We now want to be able to take expectations over $\omega'\in\Omega$, using our usual probability measure under which each bit of $\omega'$ is $1$ with probability $p$ and $0$ with probability $1-p$, while keeping $\omega$ and $\tau$ fixed. We write $\Ps$ to emphasise that $\omega$ and $\tau$ are fixed. The notation $\widehat{h_{J}}(S)$ will mean the Fourier coefficient with respect to $\Ps$, i.e.~$\Es[h_\Js(\omega')\chi_S(\omega')]$. The set $J$ will always be a function of $\omega$ and $\tau$, $J=J(\omega,\tau)$ (and not $\omega'$), though we will omit this from the notation for the sake of readability.

  We start with a general lemma about any such function $h$, before choosing a particular $h$. We stress that these proofs are almost identical to those in \cite{MR2630053}, but fleshed out and adapted to our more general situation.

  \begin{lemma}\label{lem:hgivenA}
    For any $S\subset E$ and any function $h:\Omega\to\R$,
    \[\Et[h(\omega)|\A] = \widehat{h_{J}}(\emptyset).\]
  \end{lemma}

  \begin{proof}
    Setting $\omega^S$ to be $1$ on $S$ and $0$ off $S$, we have
    \begin{align*}
      \widehat{h_{J}}(\emptyset) &= \Es[h_{J}(\omega')]\\
      &= \sum_{S\subset E} h_{J}(\omega^S) p^{|S|}(1-p)^{|E\setminus S|}\\
      &= \sum_{S\subset E} h(\omega^S_J) p^{|S|}(1-p)^{|E\setminus S|}\\
      &= \sum_{S\subset J^c} h(\omega^{S\cup J'}) p^{|S|}(1-p)^{|J^c\setminus S|}
    \end{align*}
    where $J' = J'(\omega,\tau)=\{e\in J(\omega,\tau) : \omega(e)=1\}$. But this last quantity is exactly $\Et[h(\omega)|\A]$.
  \end{proof}

  We now fix a function $h$ by setting
  \begin{equation}\label{eq:chose_h}h(\omega) = \sum_{\substack{|S|=k,\\ S\cap U=\emptyset}} \hat f(S) \chi_S(\omega).\end{equation}

  \begin{lemma}\label{lem:hJfourier}
    Suppose $h$ is as defined in~\eqref{eq:chose_h}. Then for any $S\subset E$ with $|S|=k$,
    \[\widehat{h_{J}}(S) = \begin{cases} 0 & \text{if } S\cap J\neq\emptyset\\
\hat h(S) & \text{if } S\cap J=\emptyset.\end{cases}\]
  \end{lemma}

  \begin{proof}
    Note that $h_{J}(\omega') = h(\omega'_J) = \sum_S \hat h(S)\chi_S(\omega'_J)$. Therefore
    \[\widehat{h_J}(S) = \Es[h_J(\omega')\chi_S(\omega')] = \sum_{|S'|=k} \hat h(S') \Es[\chi_{S'}(\omega'_J)\chi_S(\omega')].\]
    If $S'\neq S$, then (since $S'$ and $S$ have the same size) we may take $e\in S\setminus S'$; changing the value of the bit $e$ changes $\chi_S$ but not $\chi_{S'}$, so an easy calculation shows that in this case $\Es[\chi_{S'}(\omega'_J)\chi_S(\omega')]=0$. Thus
    \[\widehat{h_J}(S) = \Es[h_J(\omega')\chi_S(\omega')] = \hat h(S) \Es[\chi_{S}(\omega'_J)\chi_S(\omega')].\]

    Now if $S\cap J\neq\emptyset$, then we may take $e\in S\cap J$; since $e\in J$, the value of $\omega'_J$ remains constant when we change $\omega'(e)$. On the other hand, since $e\in S$, the value of $\chi_S(\omega')$ changes when we change $\omega'(e)$. Therefore another easy calculation gives that in this case also $\Es[\chi_{S}(\omega'_J)\chi_S(\omega')]=0$, and thus $\widehat{h_J}(S)=0$ when $S\cap J\neq\emptyset$.

    Finally, if $S\cap J=\emptyset$, then $\chi_{S}(\omega'_J) = \chi_S(\omega')$, so in this case by orthonormality we have $\Es[\chi_{S}(\omega'_J)\chi_S(\omega')]=1$ and $\widehat{h_J}(S)=\hat h(S)$. This completes the proof.
  \end{proof}

  \begin{lemma}\label{lem:htosum}
    For $h$ defined in~\eqref{eq:chose_h} we have that
    \[\Et[\widehat{h_{J}}(\emptyset)^2] \le \sum_{\substack{|S|=k\\ S\cap U=\emptyset}} \hat h(S)^2 \Pt(J\cap S\neq \emptyset).\]
  \end{lemma}

  \begin{proof}
    Using Plancherel's identity on the function $h_J$, we have
    \[\Es[h_J(\omega')^2] = \sum_S \widehat{h_{J}}(S)^2\]
    and therefore
    \begin{equation}\label{eq:hat_to_sum}
      \widehat{h_{J}}(\emptyset)^2 = \Es[h_J(\omega')^2] - \sum_{|S|>0} \widehat{h_{J}}(S)^2.
    \end{equation}
    If we let $g=h^2$, then applying Lemma \ref{lem:hgivenA} to $g$ and using Plancherel's identity we see that
    \[\Et[\Es[h_J(\omega')^2]] \hspace{-0.5mm} = \hspace{-0.5mm} \Et[\widehat{g_J}(\emptyset)] \hspace{-0.5mm} = \hspace{-0.5mm} \Et[\Et[g(\omega)|\A]] \hspace{-0.5mm} = \hspace{-0.5mm} \Et[g(\omega)] \hspace{-0.5mm} = \hspace{-0.5mm} \Et[h(\omega)^2] \hspace{-0.5mm} = \hspace{-0.5mm} \sum_S \hat h(S)^2.\]
    Therefore, taking expectations in \eqref{eq:hat_to_sum}, we get
    \[\Et[\widehat{h_{J}}(\emptyset)^2] = \sum_S \hat h(S)^2 - \sum_{|S|>0} \Et[\widehat{h_{J}}(S)^2].\]
    By Lemma \ref{lem:hJfourier}, $\widehat{h_{J}}(S)^2 = \hat h(S)^2 \ind_{\{J\cap S=\emptyset\}}$ when $|S|=k$; and the same quantity is obviously non-negative when $|S|\neq k$, so
    \[\Et[\widehat{h_{J}}(\emptyset)^2] \le \sum_S \hat h(S)^2 - \sum_{|S|=k}\hat h(S)^2 \Pt(J\cap S=\emptyset).\]
    Since $\hat h(S)=0$ unless $|S|= k$ and $S\cap U=\emptyset$, the result follows.
  \end{proof}

  We can now prove Theorem \ref{thm:randalg}.

  \begin{proof}[Proof of Theorem \ref{thm:randalg}]
    Suppose that $h$ is as in~\eqref{eq:chose_h}. We claim first that
    \begin{equation}\label{eq:htofh}
      \Et[h(\omega)^2]^2 \le \Et[f(\omega)^2]\Et[\widehat{h_{J}}(\emptyset)^2].
    \end{equation}
    To show this, note that by orthogonality,
    \begin{align*}
      \Et[h(\omega)f(\omega)] &= \Et\Big[\sum_{\substack{|S|=k;\\ S\cap U=\emptyset}} \hat f(S)\chi_S(\omega) \sum_{S'\subset E} \hat f(S') \chi_{S'}(\omega)\Big]\\
      &= \Et\Big[\sum_{\substack{|S|=k;\\ S\cap U=\emptyset}} \hat f(S)\chi_S(\omega) \sum_{\substack{|S'|=k;\\ S'\cap U=\emptyset}} \hat f(S') \chi_{S'}(\omega)\Big] = \Et[h(\omega)^2].
    \end{align*}
    On the other hand,
    \begin{align*}
      \Et[h(\omega)f(\omega)] = \Et[\Et[h(\omega)f(\omega)|\A]] &= \Et[f(\omega)\Et[h(\omega)|\A]] \\
      &\le \Et[f(\omega)^2]^{1/2} \Et[\Et[h(\omega)|\A]^2]^{1/2}
    \end{align*}
    where the second equality uses the fact that $f$ is $\A$-measurable, and the last inequality uses Cauchy-Schwartz. Putting these two expressions for $\Et[h(\omega)f(\omega)]$ together, and recalling from Lemma \ref{lem:hgivenA} that $\Et[h(\omega)|\A] = \widehat{h_{J}}(\emptyset)$, we get \eqref{eq:htofh}.

    Now, combining \eqref{eq:htofh} with Lemma \ref{lem:htosum},
    \[\Et[h(\omega)^2]^2 \le \Et[f(\omega)^2]\sum_{\substack{|S|=k,\\ S\cap U=\emptyset}} \hat h(S)^2 \Pt(J\cap S\neq\emptyset).\]
    Taking a union bound, for any $S$ with $|S|=k$ and $S\cap U=\emptyset$ we have $\Pt(J\cap S \neq\emptyset) \le k \mathcal R_U$, so
    \[\Et[h(\omega)^2]^2 \le \Et[f(\omega)^2]\sum_{\substack{|S|=k;\\ S\cap U=\emptyset}} \hat h(S)^2 k \mathcal R_U.\]
    By Plancherel's identity and the definition of $h$,
    \begin{equation}\label{eq:plan_h}
      \sum_{\substack{|S|=k;\\ S\cap U=\emptyset}} \hat h(S)^2 = \sum_S \hat h(S)^2 = \Et[h(\omega)^2],
    \end{equation}
    so
    \[\Et[h(\omega)^2]^2 \le \Et[f(\omega)^2] \Et[h(\omega)^2] k \mathcal R_U\]
    and therefore $\Et[h(\omega)^2] \le \Et[f(\omega)^2] k \mathcal R_U$. Since $\hat h(S) = \hat f(S)$ for all $S$ with $|S|=k$ and $S\cap U=\emptyset$, using \eqref{eq:plan_h} again we have
    \[\sum_{\substack{|S|=k;\\ S\cap U=\emptyset}} \hat f(S)^2 = \Et[h(\omega)^2] \le \Et[f(\omega)^2] k \mathcal R_U.\qedhere\]
  \end{proof}

  \subsection{Pivotality}\label{sec:piv}

  In Section \ref{sec:randalg} we gave a method for bounding
  \[\sum_{\substack{|S|=k,\\ S\cap U=\emptyset}}\hat f(S)^2,\]
  which we will apply by fixing a vertex $v$ and letting $U$ be the set of edges that do not have an endpoint at $v$. In this section we will give a bound on the Fourier coefficients of sets that \emph{do} contain a particular edge, using the notion of pivotality.

  An edge $e\in E$ is said to be \emph{pivotal} for $f$ and $\omega\in \Omega$ if $f(\sigma_e(\omega))\neq f(\omega)$, where $\sigma_e(\omega)$ is the configuration obtained from $\omega$ by switching the value of $\omega(e)$. Let $\mathcal P_f=\mathcal P_f(\omega)$ denote the set of pivotal edges.
  The next lemma allows us to control the Fourier coefficients by estimating the probability of being pivotal.
  Similar results are known in the case when $\P$ is a uniform measure; see~\cite[Proposition 4.4 and Chapter 9]{garban_steif_noise}.
  The non-uniform case is somewhat more delicate.
  
  We say that two functions $f,g:\Omega\to\R$ are \emph{jointly monotone} if
  \[\big(f(\omega)-f(\sigma_e(\omega))\big)\big(g(\omega)-g(\sigma_e(\omega))\big)\ge 0 \hspace{3mm} \forall e\in E.\]
  In particular if $f$ and $g$ are both monotone increasing (or both monotone decreasing) then $f$ and $g$ are jointly monotone.

  \begin{lemma}\label{lemma:fourier_to_pivotal}
    Suppose that $f,g:\Omega\rightarrow\{0,1\}$ are jointly monotone. Then for any $e\in E$,
    \[
      \sum_{S: e\in S}\hat f(S) \hat g(S) = p(1-p)\P(e\in\mathcal P_f\cap \mathcal P_g).
    \]
  \end{lemma}
  
  \begin{proof}
    Fix $e\in E$ and define an operator $\nabla_e$ by setting
    \[
      \nabla_e f(\omega) = |r_e(\omega)|\big(f(\omega)-f(\sigma_e(\omega))\big).
    \]
    Since $f(\omega)=\sum_S \hat f(S)\chi_S(\omega)$, from the definition of $\chi_S$ we have that
    \[
      \nabla_e f(\omega) = |r_e(\omega)|\big(r_e(\omega)-r_e(\sigma_e(\omega))\big)\sum_{S:e\in S}\hat f(S)\chi_{S\setminus\{e\}}(\omega)
    \]
    Now, if $\omega(e)=1$, then $r_e(\omega)=((1-p)/p)^{1/2}$ and $r_e(\sigma_e(\omega))=-(p/(1-p))^{1/2}$ and so
    \begin{align*}
      |r_e(\omega)|\big(r_e(\omega)-r_e(\sigma_e(\omega))\big)&= \left(\frac{1-p}{p}\right)^{1/2}\left(\left(\frac{1-p}{p}\right)^{1/2}+\left(\frac{p}{1-p}\right)^{1/2}\right)\\
      &=1/p\\
      &=\frac{r_e(\omega)}{p^{1/2}(1-p)^{1/2}}.
    \end{align*}
    On the other hand if $\omega(e)=0$, then $r_e(\omega)=-(p/(1-p))^{1/2}$ and $r_e(\sigma_e(\omega))=((1-p)/p)^{1/2}$ so that
    \begin{align*}
      |r_e(\omega)|\big(r_e(\omega)-r_e(\sigma_e(\omega))\big)&= \left(\frac{p}{1-p}\right)^{1/2}\left(-\left(\frac{p}{1-p}\right)^{1/2}-\left(\frac{1-p}{p}\right)^{1/2}\right)\\
      &=-1/(1-p)\\
      &=\frac{r_e(\omega)}{p^{1/2}(1-p)^{1/2}}.
    \end{align*}
    Thus either way, we see that
    \[
      \nabla_e f(\omega) = \frac{1}{p^{1/2}(1-p)^{1/2}}\sum_{S:e\in S}\hat f(S)\chi_{S}(\omega).
    \]
    It follows that
    \[
      \widehat{\nabla_e f}(S)=\begin{cases}
        p^{-1/2}(1-p)^{-1/2}\hat f(S) &\text{if }e\in S\\
        0&\text{if }e\notin S
      \end{cases}
    \]
    and by Plancherel's identity~\eqref{eq:plancherel},
    \begin{equation}\label{eq:nabla_1}
      \E\left[(\nabla_e f)(\nabla_e g)\right]=\sum_S \widehat{\nabla_e f}(S)\widehat{\nabla_e g}(S)=\frac{1}{p(1-p)}\sum_{S: e \in S}\hat f(S)\hat g(S).
    \end{equation}

    Next we compute $\E\left[(\nabla_e f)(\nabla_e g)\right]$ directly.
    Notice that since $f$ and $g$ are jointly monotone,
    \[
      \nabla_e f(\omega)\nabla_e g(\omega)=\begin{cases}
        (1-p)/p &\text{if }e\in\mathcal P_f(\omega)\cap \mathcal P_g(\omega)\text{ and }\omega(e)=1\\
        p/(1-p) &\text{if }e\in\mathcal P_f(\omega)\cap \mathcal P_g(\omega)\text{ and }\omega(e)=0\\
        0&\text{otherwise.}
      \end{cases}
    \]
    Since the event $\{e\in\mathcal P_f\cap \mathcal P_g\}$ is independent of $\omega(e)$, we see that
    \begin{align*}
      \E\left[(\nabla_e f)(\nabla_e g)\right]&=p\frac{1-p}{p}\P(e\in\mathcal P_f\cap \mathcal P_g) + (1-p)\frac{p}{1-p}\P(e\in\mathcal P_f\cap \mathcal P_g)\\
      & = \P(e\in\mathcal P_f\cap \mathcal P_g).
    \end{align*}
    The lemma now follows by combining this with~\eqref{eq:nabla_1}.
  \end{proof}

  \section{Component sizes of \texorpdfstring{Erd\H{o}s-R\'enyi}{Erdos-Renyi} graphs}\label{sec:prelim}
  In this section we collect some preliminary results about component sizes for Erd\H{o}s-R\'enyi graphs, which will be useful later on. We let $\mathbf P_{n,p}$ be the law of ER$(n,p)$, $\C_v$ the connected component containing vertex $v$, and $L_n$ the size of the largest connected component.
   
  We begin by presenting a result that gives the tail behaviour of the size of components. For a proof of Proposition \ref{prop:pittel}, see \cite{roberts:ER}. Pittel \cite[Proposition 2]{MR1842114} proved part (b) when $\lambda$ is fixed and $k=an^{2/3}$ where $a$ is large but does not depend on $n$.

  \begin{proposition}\label{prop:pittel}
    Let $G$ be an ER$(n,1/n-\lambda n^{-4/3})$ random graph. Write $p=1/n-\lambda n^{-4/3}$. Suppose that $(3\lambda\wedge 1)\le A_n \ll n^{1/12}$ and $|\lambda| \ll n^{1/12}$.
    Then as $n\to\infty$,
    \begin{enumerate}[(a)]
      \item For any vertex $v$, 
        \[
        \displaystyle \mathbf P_{n,p}(|\mathcal C_v| \ge A_n n^{2/3} ) = \frac{A_n^{3/2}}{(8\pi)^{1/2} n^{1/3}G'_\lambda(A_n)} e^{-G_\lambda(A_n)} (1+O(\tfrac{1}{A_n})+o(1));
      \]
      \item $\displaystyle \mathbf P_{n,p}\left(L_n \ge A_n n^{2/3} \right) = \frac{A_n^{1/2}}{(8\pi)^{1/2} G'_\lambda(A_n)} e^{-G_\lambda(A_n)} (1+O(\tfrac{1}{A_n})+o(1))$
    \end{enumerate}
    where $G_\lambda(x) = x^3/8 - \lambda x^2/2 + \lambda^2 x/2$.
  \end{proposition}

  We will also need bounds on $\mathbf P_{n,p}(|\mathcal C_v|=k)$. Again we refer to \cite{roberts:ER} for a proof.

  \begin{lemma}\label{lemma:branching_approx}
    Let $G=(V,E)$ be an ER$(n,1/n-\lambda_n n^{-4/3})$ random graph. Let $p=1/n-\lambda_n n^{-4/3}$ and fix $M\in(0,\infty)$. There exist constants $0< c_1\le c_2<\infty$ such that
    \begin{enumerate}[(a)]
      \item if $k\le Mn^{2/3}$ and $|\lambda_n| \le n^{1/12}$, then for any vertex $v$,
        \[\frac{c_1}{k^{3/2}}e^{-F_\lambda(k/n^{2/3})} \le \mathbf P_{n,p}(|\C_v|=k) \le \frac{c_2}{k^{3/2}}e^{-F_\lambda(k/n^{2/3})}\]
        where $F_\lambda(x) = x^3/6 - \lambda x^2/2 + \lambda^2 x/2$;
      \item if $n^{2/3}\le k\le n^{3/4}$ and $|\lambda| \le n^{1/12}$, then for any vertex $v$,
        \[\frac{c_1 k^{3/2}}{n^2} e^{-G_\lambda(k/n^{2/3})} \le \mathbf P_{n,p}(|\C_v|=k) \le \frac{c_2 k^{3/2}}{n^2} e^{-G_\lambda(k/n^{2/3})}\]
        where $G_\lambda(x) = x^3/8 - \lambda x^2/2 + \lambda^2 x/2$.
    \end{enumerate}
  \end{lemma}
  
  Adapting these bounds for our particular purposes, we get the following.
  
  \begin{lemma}\label{lemma:uvjk}
  There exists a finite constant $c$ such that whenever $0\le k\ll n^{3/4}$, for any vertex $v$,
  \begin{enumerate}[(a)]
      \item for any $j\ge 1$,
        \[\mathbf P_{n-k,1/n}(|\C_v|\ge j) \le \frac{c}{j^{1/2}}\exp\Big(-\frac{(k+j)^3}{8n^2} + \frac{k^3}{8n^2}\Big);\]
      \item if $(n-k)^{2/3}\le j \ll (n-k)^{3/4}$ then
        \[\mathbf P_{n-k,1/n}(|\C_v|= j) \le \frac{c j^{3/2}}{n^2} \exp\Big(-\frac{(k+j)^3}{8n^2} + \frac{k^3}{8n^2}\Big).\]
    \end{enumerate}
  \end{lemma}
  
  \begin{proof}
    Note that
    \[\frac{1}{n} = \frac{1}{n-k} - \frac{1}{(n-k)^{4/3}} \Big(\frac{(n-k)^{1/3}k}{n}\Big).\]
    Therefore, setting $\lambda = - \frac{(n-k)^{1/3}k}{n}$ and $p=1/(n-k)-\lambda (n-k)^{-4/3}$ and applying Lemma \ref{lemma:branching_approx}(b), for $j\ge (n-k)^{2/3}$ we get
    \[\mathbf P_{n-k,1/n}(|\C_v|=j) = \mathbf P_{n-k,p}(|\C_v|=j) \le \frac{c j^{3/2}}{(n-k)^2} e^{-G_\lambda(j/(n-k)^{2/3})}.\]
    Similarly, noting that for $\lambda\le 0$ we have $G'_\lambda(x)\ge 3x^2/8$,  by Proposition \ref{prop:pittel}(a), if $j\ge (n-k)^{2/3}$ then
    \[\mathbf P_{n-k,1/n}(|\C_v|\ge j) \le \frac{c'}{j^{1/2}} e^{-G_\lambda(j/(n-k)^{2/3})}.\]
    Thirdly, since $F_\lambda(x)\ge G_\lambda(x)$ for all $x\ge 0$ and $F_\lambda$ is increasing in $x$, by Lemma \ref{lemma:branching_approx}(b), if $j\le (n-k)^{2/3}$ then
    \[\sum_{i=j}^{(n-k)^{2/3}} \mathbf P_{n-k,1/n}(|\C_v|=j) \le \frac{c''}{j^{1/2}} e^{-F_\lambda(j/(n-k)^{2/3})} \le \frac{c''}{j^{1/2}} e^{-G_\lambda(j/(n-k)^{2/3})}.\]
    It therefore remains to show that
    \[G_\lambda\Big(\frac{j}{(n-k)^{2/3}}\Big) \ge \frac{(k+j)^3}{8n^2} - \frac{k^3}{8n^2}.\]
    But indeed
    \begin{align*}
      \frac{k^3}{8n^2} + G_\lambda\Big(\frac{j}{(n-k)^{2/3}}\Big) &= \frac{k^3}{8n^2} + \frac{j^3}{8(n-k)^2} + \frac{j^2 k}{2n(n-k)} + \frac{j k^2}{2n^2}\\
      &\ge \frac{k^3}{8n^2} + \frac{j^3}{8n^2} + \frac{j^2 k}{2n^2} + \frac{j k^2}{2n^2}\\
      &\ge \frac{(k+j)^3}{8n^2}
    \end{align*}
    and the result follows.
  \end{proof}

  We give two more lemmas, which follow fairly easily from those above, but are less obviously useful. We will see later that they are exactly the bounds we need to estimate the probability that two vertices have unusually large components at different times.

  \begin{lemma}\label{lemma:Cuandv}
    Fix $M>0$. There exists a finite constant $c$ such that if $2n^{2/3}\le N\ll n^{3/4}$ then
    \[\mathbf P_{n,1/n} \big( |\C_u\cup \C_v|\ge N, \, |\C_u|< N, \, |\C_v| < N, \, \C_u\cap \C_v = \emptyset\big) \le c \frac{N^2}{n^2} e^{-N^3/(8n^2)}.\]
  \end{lemma}

  \begin{proof}
    Clearly
    \begin{align*}
      &\mathbf P_{n,1/n} \big( |\C_u\cup \C_v|\ge N, \, |\C_u|< N, \, |\C_v| < N, \, \C_u\cap \C_v = \emptyset\big)\\
      &\quad\le 2 \mathbf P_{n,1/n} \big( |\C_u\cup \C_v|\ge N, \, |\C_v|\le |\C_u|< N, \, \C_u\cap \C_v = \emptyset\big)\\
      &\quad\le 2 \sum_{k=\lceil N/2\rceil}^{N-1} \mathbf P_{n,1/n}\big( |\C_u|=k \big) \mathbf P_{n,1/n} \big( |\C_v|\ge N-k, \, \C_u\cap \C_v = \emptyset \, \big|\, |\C_u|=k \big)\\
      &\quad= 2 \sum_{k=\lceil N/2\rceil}^{N-1} \mathbf P_{n,1/n}\big( |\C_u|=k \big) \mathbf P_{n-k,1/n} \big( |\C_v|\ge N-k \big).
    \end{align*}
    Applying Lemmas \ref{lemma:branching_approx}(b) and \ref{lemma:uvjk}(a), for $2n^{2/3}\le N\ll n^{3/4}$
    \begin{multline*}
    \mathbf P_{n,1/n} \big( |\C_u\cup \C_v|\ge N, \, |\C_u|< N, \, |\C_v| < N, \, \C_u\cap \C_v = \emptyset\big)\\
    \le 2 \hspace{-2.5mm}\sum_{k=\lceil N/2\rceil}^{N-1} \hspace{-2.5mm} \frac{ck^{3/2}}{n^2}e^{-k^3/(8n^2)}\frac{c'}{(N-k)^{1/2}}e^{-N^3/(8n^2)+k^3/(8n^2)} \hspace{-0.2mm}\le\hspace{-0.2mm} c'' \frac{N^2}{n^{2}}e^{-N^3/(8n^2)}.
    \end{multline*}
    This completes the proof.
  \end{proof}

  \begin{lemma}\label{lemma:Cuvw}
    There exists a finite constant $c$ such that for any distinct vertices $u$, $v$ and $w$, if $N \ll n^{3/4}$,
    \begin{align*}
      &\mathbf P_{n,1/n} \big( |\C_u\cup \C_v|\ge N, \, |\C_u|< N,\, \C_u\cap \C_v = \emptyset,\, w\in \C_u\big) \\
      &\hspace{14em}\le c \Big(\frac{1}{n^{2/3}N^{1/2}} + \frac{N^3}{n^3}\Big) e^{-N^3/(8n^2)}.
  \end{align*}
  \end{lemma}
  
  \begin{proof}
    We begin by summing over the possible sizes for $\C_u$:
    \begin{align*}
    &\mathbf P_{n,1/n} \big( |\C_u\cup \C_v|\ge N, \, |\C_u|< N,\, \C_u\cap \C_v = \emptyset,\, w\in \C_u\big)\\
    &\hspace{25mm}= \sum_{j=2}^{N-1} \mathbf P_{n,1/n} \big( |\C_u|=j,\, |\C_v|\ge N-j,\, \C_u\cap \C_v = \emptyset,\, w\in \C_u\big)\\
    &\hspace{25mm}\le \sum_{j=2}^{N-1} \frac{j}{n} \mathbf P_{n,1/n} \big( |\C_u|=j \big) \mathbf P_{n-j,1/n} \big( |\C_v|\ge N-j \big).
    \end{align*}
    Write $K = \lfloor n^{2/3}\rfloor \wedge (N-1)$. For those values of $j$ less than $K$, by Lemmas \ref{lemma:branching_approx}(a) and \ref{lemma:uvjk}(a),
    \begin{align*}
      &\sum_{j=2}^{K} \frac{j}{n} \mathbf P_{n,1/n} \big( |\C_u|=j \big) \mathbf P_{n-j,1/n} \big( |\C_v|\ge N-j\big) \\
      &\hspace{25mm}\le c\sum_{j=2}^K \frac{j}{n}\frac{1}{j^{3/2}} \frac{1}{(N-j)^{1/2}} e^{-N^3/(8n^2)+j^3/(8n^2)}\\
      &\hspace{25mm}\le c' \frac{1}{n^{2/3}N^{1/2}} e^{-N^3/(8n^2)}.
    \end{align*}
    On the other hand, for those values of $j$ between $K$ and $N-1$, by Lemmas \ref{lemma:branching_approx}(b) and \ref{lemma:uvjk}(a),
    \begin{multline*}
    \sum_{j=K+1}^{N-1} \frac{j}{n} \mathbf P_{n,1/n} \big( |\C_u|=j \big) \mathbf P_{n-j,1/n} \big( |\C_v|\ge N-j\big)\\
    \le c\sum_{j=K+1}^{N-1} \frac{j^{5/2}}{n^3}e^{-j^3/(8n^2)}\frac{1}{(N-j)^{1/2}}e^{-N^3/(8n^2) + j^3/(8n^2)} \le c'\frac{N^3}{n^3}e^{-N^3/(8n^2)}
    \end{multline*}
    as required.
  \end{proof}

  \section{Exceptional times exist for \texorpdfstring{$\beta<2/3$}{small beta}}\label{sec:small_beta}
  In this section we aim to show that if $\beta<2/3$, then with high probability there exist times $t\in[0,1]$ when $|L_n(t)|> \beta n^{2/3}\log^{1/3}n$.
  Let $I=[\beta n^{2/3}\log^{1/3} n, 2 \beta n^{2/3}\log^{1/3} n]\cap\N$ and for $v \in \{1,\dots,n\}$ let
  \[
    Z_v:=\int_0^1 \1_{\{|\mathcal C_v(t)|\in I\}}{\rm d}t.
  \]
  Then by Cauchy-Schwarz and symmetry we have that
  \begin{align}\label{eq:peyley_zygmund}
    \P\left(\sup_{t\in [0,1]}|L_n(t)|\ge \beta n^{2/3}\log^{1/3} n\right) &\geq \P\left(\sum_{v=1}^n Z_v > 0\right) \nonumber\\
    &\geq \frac{\E\left[\sum_{v=1}^n Z_v\right]^2}{\E\left[\left(\sum_{v=1}^n Z_v\right)^2\right]}\nonumber\\
    & = \frac{n^2 \E[Z_1]^2}{n\E[Z_1^2] + n(n-1)\E[Z_1Z_2]}.
  \end{align}

  We begin with a lemma which ensures that the term $n\E[Z_1^2]$ in the denominator of~\eqref{eq:peyley_zygmund} does not contribute substantially when $\beta$ is small.

  \begin{lemma}\label{lemma:expectation_squared}
    If $\beta^3 < 16/3$, then
    \[
      \lim_{n\to\infty} \frac{\E[Z^2_1]}{n\E[Z_1]^2}=0. 
    \]
  \end{lemma}
  \begin{proof}
    By Fubini's theorem, the stationarity in distribution of $\C_1(t)$, and Proposition~\ref{prop:pittel}(a) with $\lambda=0$,
    \begin{equation}\label{eq:EZasymp}
    \E[Z_1]=\int_0^1 \P(|\C_1(t)|\in I) {\rm d}t = \P(|\mathcal C_1(0))|\in I) = \frac{(1+o(1))n^{-\beta^3/8-1/3}}{((9\pi/8) \beta\log n)^{1/2}}.
  \end{equation}
  Clearly $Z_1\le 1$ so $\E[Z_1^2]\le \E[Z_1]$, so by \eqref{eq:EZasymp},
  \[
    \frac{\E[Z^2_1]}{n\E[Z_1]^2} \leq \frac{1}{n\E[Z_1]} \leq C n^{\beta^3/8-2/3} \beta^{1/2}\log^{1/2}n
  \]
  for some constant $C$. The lemma follows.
\end{proof}

Now using Lemma~\ref{lemma:expectation_squared} with~\eqref{eq:peyley_zygmund}, it remains to show that
\[
  \limsup_{n\to\infty}\frac{\E[Z_1Z_2]}{\E[Z_1]^2}\leq 1.
\]
Notice that by Fubini's theorem,
\begin{equation}\label{eq:Z_1Z_2_by_integral}
\E[Z_1Z_2]=\int_0^1\int_0^1 \P(|\mathcal C_1(s))|\in I;|\mathcal C_2(t)|\in I)\,{\rm d}t\,{\rm d}s.
\end{equation}

We will estimate the double integral on the right hand side of~\eqref{eq:Z_1Z_2_by_integral} by splitting it into two pieces. We begin with an estimate for when $|t-s|$ is small.

\subsection{Small \texorpdfstring{$|t-s|$}{t-s}: a combinatorial method}

\begin{lemma}\label{lemma:smallt}
  Let $P=\P(|\mathcal C_v|\geq \beta n^{2/3}\log^{1/3} n)$. Then for any $\delta>0$,
  \begin{multline*}
    \int_0^1 \int_0^1 \P(|\mathcal C_1(s)|\in I; |\mathcal C_2(t)|\in I) \ind_{\{|t-s|\le\delta\}} \,{\rm d}t\, {\rm d}s \\
    \leq 2\delta P^2 +\frac{4\beta\log^{1/3} n}{n^{1/3}}\delta P + 2\delta^2 P.
  \end{multline*}
\end{lemma}
\begin{proof}
  First note that, by stationarity,
  \begin{multline}\label{twointstoone}
    \int_0^1 \int_0^1 \P(|\mathcal C_1(s)|\in I; |\mathcal C_2(t)|\in I) \ind_{\{|t-s|\le\delta\}}\, {\rm d}t\, {\rm d}s \\
    \le 2\int_0^\delta \P(|\mathcal C_1(0)|\in I; |\mathcal C_2(t)|\in I)\, {\rm d}t.
  \end{multline}
  Now fix $\delta\in[0,1]$ and let $t \in [0,\delta]$.
  We partition $\P(|\mathcal C_1(0)|\in I; |\mathcal C_2(t)|\in I)$ into three cases and analyse each case separately.
  Recall that $\mathbf P_{n,p}$ denotes the law of an Erd\H{o}s-R\'enyi graph ER$(n,p)$.

  First consider the case when $|\mathcal C_1(0)\cap \mathcal C_2(t)|=0$.
  Then
  \begin{align}\label{eq:bk_disjoint}
    &\P(|\mathcal C_1(0)|\in I; |\mathcal C_2(t)|\in I ; |\mathcal C_1(0)\cap \mathcal C_2(t)|=0) \nonumber\\
    &=\sum_{k\in I} \Big(\P\big(|\mathcal C_1(0)| = k ; |\mathcal C_1(0)\cap \mathcal C_2(t)|=0\big) \nonumber\\
    & \hspace{20mm}\times\P\big( |\mathcal C_2(t)|\in I \,\big|\, |\mathcal C_1(0)|=k; |\mathcal C_1(0)\cap \mathcal C_2(t)|=0\big)\Big)\nonumber\\
    &\leq\sum_{k\in I} \mathbf P_{n,1/n}(|\mathcal C_1|=k)\mathbf P_{n-k,1/n}(|\mathcal C_2|\in I)\nonumber\\
    &\leq\sum_{k\geq n^{2/3}\log^\beta n} \mathbf P_{n,1/n}(|\mathcal C_1|=k)\mathbf P_{n-k,1/n}(|\mathcal C_2|\geq \beta n^{2/3}\log^{1/3} n)\nonumber\\
    &\leq P^2
  \end{align}
  where in the final inequality we have used the monotonicity of the event $\{|\mathcal C_2|\geq \beta n^{2/3}\log^{1/3} n\}$ in the number of vertices of the graph.

  The second case that we look at is when $2 \in \mathcal C_1(0)$.
  In this case
  \begin{align}\label{eq:2_in_C0}
    &\P\big(|\mathcal C_1(0)|\in I;|\mathcal C_2(t)|\in I; 2\in \mathcal C_1(0)\big)\nonumber \\
    &\hspace{25mm}\leq \P\big(|\mathcal C_1(0)|\in I\big)\P\big(2\in \mathcal C_1(0)\,\big|\, |\mathcal C_1(0)|\in I\big) \nonumber\\
    & \hspace{25mm}\leq P\frac{2\beta n^{2/3}\log^{1/3} n}{n}.
  \end{align}

  Finally we are left to estimate the probability of the event
  \[\mathcal E:=\{|\mathcal C_1(0)|\in I;|\mathcal C_2(t)|\in I;|\mathcal C_1(0)\cap \mathcal C_2(t)|>0 ; 2 \notin \mathcal C_1(0)\}.\]
  Take $A \subset \{1,\dots,n\}$ such that $|A|\in I$, and condition on $\mathcal C_1(0)=A$.
  On the event $\mathcal E$, there exists at least one open path at time $t$ between $A$ and the vertex $2$.
  Let $\pi$ be the shortest such path (chosen arbitrarily in the case of a tie). Then on $\mathcal E$,
\begin{list}{(\roman{enumi})}{\usecounter{enumi}}
  \item $\pi$ starts at a vertex $v\in A$ and ends at the vertex $2$,
  \item $\pi$ first crosses an edge connecting $A$ to $A^c$, and otherwise only uses edges with both end points in $A^c$,
  \item all of the edges in $\pi$ are open at time $t$.
  \end{list}
  We now estimate the probability of a path satisfying (i), (ii) and (iii) existing.

  There are at most $2\beta n^{2/3}\log^{1/3} n$ vertices in $A$, and at most $n-\beta n^{2/3}\log^{1/3} n$ vertices in $A^c$, so the number of paths of length $k$ satisfying (i) and (ii) is at most $(2\beta n^{2/3}\log^{1/3} n)(n-\beta n^{2/3}\log^{1/3} n)^{k-1}$.
  Under the conditioning $\mathcal C_1(0)=A$, every edge $e$ with both end points lying in $A^c$ is open at time $t$ with probability $1/n$.
  Moreover, any edge $e'$ with one end point in $A$ and the other in $A^c$ is open at time $t$ with probability $(1-e^{-t})/n$: we know that at time $0$ the edge $e'$ is closed (since $A$ is not connected to $A^c$), and thus in order for it to be open at time $t$ we must first resample the edge, and then open the edge at the resampling.
  Thus in conclusion we see that the probability there exists a path $\pi$ of length $k$ satisfying (i), (ii) and (iii) is at most
  \begin{multline*}
    (2\beta n^{2/3}\log^{1/3} n)(n-\beta n^{2/3}\log^{1/3} n)^{k-1}\cdot \frac{1}{n^{k-1}}\cdot \frac{1-e^{-t}}{n}\\
    \leq 2t(1-\beta n^{-1/3}\log^{1/3} n)^{k-1} \beta n^{-1/3} \log^{1/3} n 
  \end{multline*}
  where for the inequality we have used the fact that $1-e^{-t}\leq t$.
  Summing over $k$, we see that the probability there exists a path $\pi$ satisfying (i), (ii) and (iii) is at most
  \[
    2t \beta n^{-1/3}\log^{1/3} n \sum_{k=1}^\infty (1-\beta n^{-1/3}\log^{1/3} n)^{k-1} = 2t.
  \]
  Hence we obtain
  \begin{equation}\label{eq:2notin}
    \P(|\mathcal C_1(0)|\in I; |\mathcal C_2(t)|\in I; |\mathcal C_1(0)\cap \mathcal C_2(t)|>0; 2 \notin \mathcal C_1(0)) \leq 2t.
  \end{equation}

  Putting together~\eqref{eq:bk_disjoint},~\eqref{eq:2_in_C0} and~\eqref{eq:2notin} we get
  \[
    \P(|\mathcal C_1(0)|\in I; |\mathcal C_2(t)|\in I) \leq P^2 + P\frac{2\beta n^{2/3}\log^{1/3} n}{n} + 2t.
  \]
  Integrating over $t\in[0,\delta]$ and using \eqref{twointstoone} gives the desired result.
\end{proof}

\subsection{Large \texorpdfstring{$|t-s|$}{t-s}: applying Fourier analysis}\label{sec:largets}

Fix $\delta>0$. Our next aim is to estimate the integral
\[
  \int_0^1\int_0^1 \P(|\mathcal C_1(s)|\in I; |\mathcal C_2(t)|\in I)\ind_{\{|t-s|>\delta\}}\, {\rm d}t\, {\rm d}s.
\]
To do this, we will use the Fourier analysis introduced in Section~\ref{sec:fourier}.

Fix $N\in\N$. For a vertex $v \in \{1,\dots,n\}$, let $f_v:\Omega\rightarrow \{0,1\}$ be the function given by
\[
  f_v(\omega)=\begin{cases}
    1 & \text{if the connected component of }v\text{ in }\omega\text{ has size at least }N\\
    0 & \text{otherwise.}
  \end{cases}
\]

We recall some notation from Section~\ref{sec:fourier}. For $\omega \in \Omega$ and $\eps\in [0,1]$, let $\omega_\eps$ be the random configuration obtained from $\omega$ by resampling each edge in $\omega$ with probability $\eps$. Lemma \ref{lemma:noise_fourier} told us that
\begin{equation}\label{eq:noise_fourier}
  \E[f_1(\omega)f_2(\omega_\eps)]=\sum_{S} \hat f_1(S)\hat f_2(S) (1-\eps)^{|S|}.
\end{equation}
In our setting of the dynamical Erd\H{o}s-R\'enyi graph, the configuration at time $t>s$ can be obtained from the configuration at time $s$ by resampling each edge with probability $\eps = 1-e^{-(t-s)}$.
Hence for any $\delta\in (0,1)$, if $N=\lceil \beta n^{2/3}\log^{1/3}n\rceil$,
\begin{align}\label{eq:int_to_fourier}
  &\int_0^1\int_0^1 \P(|\mathcal C_1(s)|\in I; |\mathcal C_2(t)|\in I)\ind_{\{|t-s|>\delta\}}\, {\rm d}t\, {\rm d}s\nonumber\\
  &\hspace{30mm} \le \int_0^1\int_0^1 \P(|\mathcal C_1(s)|\ge N; |\mathcal C_2(t)|\ge N)\ind_{\{|t-s|>\delta\}} \,{\rm d}t \,{\rm d}s\nonumber\\
  &\hspace{30mm}= \int_0^1\int_0^1 \E[f_1(\omega)f_2(\omega_{1-e^{-|t-s|}})] \ind_{\{|t-s|>\delta\}} \,{\rm d}t \,{\rm d}s\nonumber\\
  &\hspace{30mm}=\sum_{S} \hat f_1(S)\hat f_2(S) \int_0^1\int_0^1 e^{-|t-s| |S|}\ind_{\{|t-s|>\delta\}} \,{\rm d}t \,{\rm d}s\nonumber\\
  &\hspace{30mm}\leq \hat f_1(\emptyset)\hat f_2(\emptyset) + \sum_{|S|>0}\hat f_1(S) \hat f_2(S)\cdot 2\int_\delta^1 e^{-t|S|}\, {\rm d}t \nonumber\\
  &\hspace{30mm}\leq \hat f_1(\emptyset)^2 + 2\sum_{|S|>0}\frac{e^{-\delta |S|}}{|S|}\hat f_1(S)\hat f_2(S).
\end{align}

Let $\mathcal U_v$ be the set of edges that have an end point at $v$.
We will study the Fourier coefficients $\hat f_1(S)\hat f_2(S)$ by separating into cases when $S\cap (\mathcal U_1\cup\,\mathcal U_2)\neq\emptyset$ and when $S\cap (\mathcal U_1\cup\,\mathcal U_2)=\emptyset$. For the former case we will apply Lemma \ref{lemma:fourier_to_pivotal}, and for the latter we will use Theorem \ref{thm:randalg}.
We begin by studying the former. 

\begin{lemma}\label{lemma:U1U2intersect}
  Let $N=\lceil \beta n^{2/3}\log^{1/3}n\rceil$. Then there exists a finite constant $C$ such that
  \[
    \sum_{S:S \cap (\mathcal U_1\cup\mathcal U_2)\neq \emptyset}\hat f_1(S)\hat f_2(S)\leq C (\beta^2+\beta^{-1/2}) n^{-1-\beta^3/8}\log n.
  \]
\end{lemma}
\begin{proof}
  The two functions $f_1$ and $f_2$ are both increasing and therefore jointly monotone (see the definition before Lemma \ref{lemma:fourier_to_pivotal}). Therefore, by Lemma~\ref{lemma:fourier_to_pivotal}, we have
  \begin{align}
    \sum_{S:S \cap (\mathcal U_1\cup\mathcal U_2)\neq \emptyset}&\hat f_1(S)\hat f_2(S)\nonumber\\
    &\leq \sum_{S:(1,2)\in S}\hat f_1(S)\hat f_2(S) + \sum_{v=3}^n \sum_{S:(1,v)\in S}\hat f_1(S)\hat f_2(S) \nonumber\\
    &\qquad+ \sum_{v=3}^n \sum_{S:(2,v)\in S}\hat f_1(S)\hat f_2(S)\nonumber\\
    &\leq \frac{1}{n}\left(1-\frac{1}{n}\right)\P((1,2)\in \mathcal P_{f_1}\cap \mathcal P_{f_2}) \nonumber\\
    &\qquad+ 2(n-2)\cdot \frac{1}{n}\left(1-\frac{1}{n}\right)\max_{u\in\{1,2\}, v\neq 1,2}\P((u,v)\in \mathcal P_{f_1}\cap \mathcal P_{f_2})\nonumber\\
    &\leq \frac{1}{n}\P((1,2)\in\mathcal P_{f_1}\cap\mathcal P_{f_2}) + 2\P((1,3)\in \mathcal P_{f_1}\cap\mathcal P_{f_2}).\label{eq:fhattoPbd}
  \end{align}
  We first bound $\P((1,2)\in\mathcal P_{f_1}\cap\mathcal P_{f_2})$. Since the event that $(1,2)$ is closed is independent of the event that $(1,2)$ is pivotal for $f_1$ and $f_2$, without loss of generality we can assume that $(1,2)$ is closed. Then for $(1,2)$ to be pivotal for both $f_1$ and $f_2$, the connected components $\mathcal C_1$ and $\mathcal C_2$ must satisfy
  \begin{enumerate}[(a)]
  \item $\mathcal C_1 \cap \mathcal C_2 =\emptyset$,
  \item $|\mathcal C_1| < N$ and $|\mathcal C_2|< N$,
  \item $|\mathcal C_1 \cup \mathcal C_2|\ge N$.
  \end{enumerate}
  That is,
  \[\P((1,2)\in\mathcal P_{f_1}\cap\mathcal P_{f_2}) \le \P\big( |\C_1\cup \C_2|\ge N, \, |\C_1| < N, \, |\C_2| < N, \, \C_1\cap \C_2=\emptyset\big).\]
  By Lemma \ref{lemma:Cuandv}, this is at most a constant times $N^2 e^{-N^3/(8n^2)}/n^2$.

  We now move on to estimating $\P((1,3)\in \mathcal P_{f_1}\cap\mathcal P_{f_2})$. Note that
  \begin{multline}\label{split13}
  \P\big((1,3)\in\mathcal P_{f_1}\cap\mathcal P_{f_2}\big)\\
  = \P\big(2\in\C_1,\,(1,3)\in\mathcal P_{f_1}\cap\mathcal P_{f_2}\big) + \P\big(2\in \C_3,\,(1,3)\in\mathcal P_{f_1}\cap\mathcal P_{f_2}\big).
  \end{multline}
  Of course,
  \begin{equation}\label{2reduce}
  \P\big(2\in\C_1,\,(1,3)\in\mathcal P_{f_1}\cap\mathcal P_{f_2}\big) = \P\big(2\in\C_1,\,(1,3)\in\mathcal P_{f_1}\big).
  \end{equation}
  Also
  \[\P\big(2\in \C_3,\,(1,3)\in\mathcal P_{f_1}\cap\mathcal P_{f_2}\big) = \P\big(2\in \C_3,\,(1,3)\in\mathcal P_{f_1}\cap\mathcal P_{f_3}\big);\]
  by symmetry, we can permute the roles of $1$ and $3$, so that
  \begin{align*}
  \P\big(2\in \C_3,\,(1,3)\in\mathcal P_{f_1}\cap\mathcal P_{f_2}\big) &= \P\big(2\in \C_1,\,(1,3)\in\mathcal P_{f_1}\cap\mathcal P_{f_3}\big)\\
  &\le \P\big(2\in \C_1,\,(1,3)\in\mathcal P_{f_1}\big)
  \end{align*}
  and therefore, combining with \eqref{split13} and \eqref{2reduce},
  \[\P\big((1,3)\in\mathcal P_{f_1}\cap\mathcal P_{f_2}\big)\le 2\P\big(2\in \C_1,\,(1,3)\in\mathcal P_{f_1}\big).\]
  
   Just as above, we may assume that $(1,3)$ is closed; and then for $(1,3)$ to be pivotal for $f_1$, the components must satisfy
  \begin{enumerate}[(i)]
  \item $\mathcal C_1 \cap \mathcal C_3 =\emptyset$,
  \item $|\mathcal C_1| < N$,
  \item $|\mathcal C_1 \cup \mathcal C_3|\ge N$.
  \end{enumerate}
  Thus
  \[\P\big((1,3)\in\mathcal P_{f_1}\cap\mathcal P_{f_2}\big)\le 2\P\big( |\C_1\cup \C_3|\ge N, \, |\C_1| < N, \, \C_1\cap \C_3=\emptyset, \, 2\in \C_1\big).\]
  Applying Lemma \ref{lemma:Cuvw}, we get
  \[\P((1,3)\in\mathcal P_{f_1}\cap\mathcal P_{f_2}) \le c \Big(\frac{1}{n^{2/3}N^{1/2}} + \frac{N^3}{n^3}\Big) e^{-N^3/(8n^2)}.\]
  for some finite constant $c$.

  Plugging these bounds back into \eqref{eq:fhattoPbd}, we have
  \[\sum_{S:S \cap (\mathcal U_1\cup\mathcal U_2)\neq \emptyset}\hat f_1(S)\hat f_2(S) \le c\Big(\frac{N^2}{n^3} + \frac{1}{N^{1/2}n^{2/3}} + \frac{N^3}{n^3} \Big)e^{-N^3/(8n^2)}.\]
  Recalling that $N=\lceil\beta n^{2/3}\log^{1/3} n\rceil$ and simplifying, we get
  \[\sum_{S:S \cap (\mathcal U_1\cup\mathcal U_2)\neq \emptyset}\hat f_1(S)\hat f_2(S) \le c'(\beta^2 +\beta^{-1/2} )n^{-1-\beta^3/8}\log n,\]
  and the result follows.
\end{proof}

Now we deal with the Fourier coefficients $\hat f_1(S)\hat f_2(S)$ where $S\cap (\mathcal U_1\cup \mathcal U_2)=\emptyset$.
Notice that by symmetry we have that if $S\cap (\mathcal U_1\cup \mathcal U_2)=\emptyset$, then $\hat f_1(S)=\hat f_2(S)$ and so
\begin{equation}\label{eq:empty_term_to_f_1}
  \sum_{|S|>0; S\cap(\mathcal U_1\cup\mathcal U_2)=\emptyset}\frac{e^{-\delta |S|}}{|S|}\hat f_1(S)\hat f_2(S)\leq\sum_{|S|>0;S\cap \mathcal U_1=\emptyset}\frac{e^{-\delta|S|}}{|S|}\hat f_1(S)^2.
\end{equation}
To estimate the sum on the right hand side we use a revealment algorithm, implementing Theorem \ref{thm:randalg}. Any sensible algorithm will do; we can reveal all of the edges emanating from vertex $1$ without concern, and thereafter the lack of geometry in the graph simplifies the problem.

The algorithm $A$ that we choose to use is the breadth first search and is described as follows.
At each step $i \geq 0$ we have an ordered list of vertices $\mathcal S_i$, which is the list of vertices that the algorithm already knows are in $\C_1$.
We begin from $\mathcal S_0=\{1\}$.
At each step $i$, if $|\mathcal S_i|\geq N$ then we terminate and declare that $f_1(\omega)=1$; or if $|\mathcal S_i|<i$ then we terminate the algorithm and declare that $f_1(\omega)=0$. Otherwise we take the $i$th element $v_i$ of $\mathcal S_i$, and reveal $\omega((v_i,w))$ for all $w\not\in\mathcal S_i$. If $\omega((v_i,w))=1$ then we add $w$ to the end of the list, and once we have revealed all such edges (in some arbitrary order), the resulting list is then $\mathcal S_{i+1}$. 

Clearly the algorithm must terminate by step $N$. Recall that
\[\mathcal R_{\mathcal U_1} = \max_{e\notin \mathcal U_1} \P(A \hbox{ reveals } \omega(e)).\]

\begin{lemma}\label{lemma:revealment}
  Let $A$ be the breadth first search described above, and let $N=\lceil \beta n^{2/3}\log^{1/3}n\rceil$. There exists a finite constant $C$ such that
  \[
    \mathcal R_{\mathcal U_1} \leq C \beta^{7/2} n^{-2/3} \log^{7/6} n.
  \]
\end{lemma}

\begin{proof}
  Let $\tau$ be the step at which the algorithm $A$ terminates. For any edge $e=(v,w)\notin\mathcal U_1$, the probability that we reveal $\omega(e)$ is at most the probability that either $v$ or $w$ appears in $\mathcal S_{\tau-1}$.
  For any $v,w \neq 1$ we have $\P(v\in \mathcal S_{\tau-1})=\P(w\in \mathcal S_{\tau-1})$, and thus
  \begin{align}
    &\P(A \text{ reveals } \omega(e)) \leq 2\P( v\in \mathcal S_{\tau-1})\nonumber \\
    &\hspace{15mm}= \frac{2}{n-1}\E \Big[\sum_{u\neq 1} \1_{\{u \in \mathcal S_{\tau-1}\}}\Big]=\frac{2(\E[|\mathcal S_{\tau-1}|]-1)}{n-1}\leq \frac{2}{n}\E[|\mathcal S_{\tau-1}|].\label{eq:reveal_expectation_estimate}
  \end{align}
  It is easy to see from the description of the algorithm that we always have $\mathcal S_{\tau-1}\subset \C_1$ and $|\mathcal S_{\tau-1}|\le N$. Combining this observation with~\eqref{eq:reveal_expectation_estimate}, then applying Proposition~\ref{prop:pittel}(a) and Lemma~\ref{lemma:branching_approx} (both with $\lambda=0$), we get that
  \begin{align*}
    \P(A\text{ reveals } \omega(e)) &\leq \sum_{k=1}^{N}\frac{2k}{n}\P(|\mathcal C_1|=k) +\frac{2N}{n}\P(|\mathcal C_1|\ge N)\\ 
    &\leq  c\sum_{k=1}^{\lfloor n^{2/3} \rfloor}\frac{k}{n}k^{-3/2} + c\sum_{k=\lfloor n^{2/3} \rfloor+1}^{N}\frac{k}{n}\frac{k^{3/2}}{n^2} + c\frac{N}{n} \frac{n^{-\beta^3/8}}{n^{1/3}\log^{1/6} n}\\ 
    & \leq c' n^{-2/3} + c'\frac{N^{7/2}}{n^3} + c' n^{-\beta^3/8-2/3}\log^{1/6} n
  \end{align*}
  for some finite constants $c,c'$. For large $n$ this is at most a constant times $\beta^{7/2} n^{-2/3}\log^{7/6}n$.
\end{proof}

Now we apply Theorem \ref{thm:randalg} and Lemma \ref{lemma:revealment} to estimate the Fourier coefficients $\hat f_1(S)\hat f_2(S)$ when $S\cap (\mathcal U_1\cup \mathcal U_2)=\emptyset$.

\begin{lemma}\label{lemma:U1U2empty}
  Let $N=\lceil \beta n^{2/3}\log^{1/3}n\rceil$. There exists a finite constant $C$ such that for any $\delta >0$,
  \[
    \sum_{\substack{|S|>0;\\S\cap (\mathcal U_1\cup\mathcal U_2)=\emptyset}}\frac{e^{- \delta |S|}}{|S|}\hat f_1(S)\hat f_2(S) \leq C\delta^{-1} \E[f_1] \beta^{7/2} n^{-2/3} \log^{7/6} n .
  \]
\end{lemma}
\begin{proof}
  First recall that by~\eqref{eq:empty_term_to_f_1} we have
  \begin{equation}\label{eq:sum_poly}
    \sum_{\substack{|S|>0;\\S\cap (\mathcal U_1\cup \mathcal U_2)=\emptyset}} \hspace{-4mm}\frac{e^{-\delta |S|}}{|S|}\hat f_1(S)\hat f_2(S) \leq \hspace{-2mm} \sum_{\substack{|S|>0;\\ S\cap \mathcal U_1=\emptyset}} \frac{e^{-\delta |S|}}{|S|}\hat f_1(S)^2 = \sum_{k=1}^{\binom{n}{2}}\frac{1}{k} e^{-\delta k}\hspace{-2mm} \sum_{\substack{|S|=k;\\ S\cap \mathcal U_1=\emptyset}}  \hat f_1(S)^2.
  \end{equation}
  By Theorem \ref{thm:randalg}, this is at most
  \[
    \sum_{k=1}^{\binom{n}{2}} e^{-\delta k} \mathcal R_{\mathcal U_1} \E[f_1^2].
  \]
  Since $f_1$ takes values in $\{0,1\}$, $\E[f_1^2]=\E[f_1]$, and by Lemma~\ref{lemma:revealment}, $\mathcal R_{\mathcal U_1} \leq C \beta^{7/2} n^{-2/3} \log^{7/6} n$.
  Finally, note that
  \[\sum_{k=1}^\infty e^{-\delta k} = \frac{e^{-\delta}}{1-e^{-\delta}} \le \delta^{-1}.\]
  Combining these three observations gives the desired result.
\end{proof}

\subsection{Completing the proof of Theorem~\ref{thm:noise} for small \texorpdfstring{$\beta$}{beta}}
We begin by recalling our argument from the start of Section \ref{sec:small_beta}. We began by defining $Z_v = \int_0^1 \ind_{\{|\C_v(t)|\in I\}}\, dt$ where $I=[\beta n^{2/3}\log^{1/3} n,2\beta n^{2/3}\log^{1/3} n]\cap\N$. From~\eqref{eq:peyley_zygmund} we know that
\[\P\Big(\sup_{t\in[0,1]} |L_n(t)| > \beta n^{2/3}\log^{1/3} n\Big) \ge \frac{n^2\E[Z_1]^2}{n\E[Z_1^2]+n(n-1)\E[Z_1 Z_2]}.\]
Lemma~\ref{lemma:expectation_squared} told us that if $\beta^3<16/3$ then $\frac{\E[Z_1^2]}{n\E[Z_1]^2}\to 0$ as $n\to\infty$, in which case we get that
\begin{equation}\label{eq:liminf_P}
  \liminf_{n\to\infty}\P \Big( \sup_{t\in [0,1]}|L_n(t)|> n^{2/3}\log^\beta n \Big)  \geq \liminf_{n\to\infty} \frac{\E[Z_1]^2}{\E[Z_1Z_2]}.
\end{equation}
We saw in~\eqref{eq:Z_1Z_2_by_integral} that
\[\E[Z_1Z_2]=\int_0^1\int_0^1 \P(|\mathcal C_1(s))|\in I;|\mathcal C_2(t)|\in I)\,{\rm d}t\,{\rm d}s.\]

Let $P = \P(|\C_1|\ge \beta n^{2/3}\log^{1/3} n)$.
Lemma \ref{lemma:smallt} gives
\begin{multline*}
  \int_0^1 \int_0^1 \P(|\mathcal C_1(s)|\in I; |\mathcal C_2(t)|\in I) \ind_{\{|t-s|\le\delta\}} \,{\rm d}t \,{\rm d}s \\
  \leq 2\delta P^2 +\frac{4\beta \log^{1/3} n}{n^{1/3}}\delta P + 2\delta^2 P.
\end{multline*}
By Proposition \ref{prop:pittel}(a) with $\lambda=0$, we have $P\le n^{-1/3-\beta^3/8}$ for large $n$, so (for large $n$)
\begin{multline}\label{firstint}
  \int_0^1 \int_0^1 \P(|\mathcal C_1(s)|\in I; |\mathcal C_2(t)|\in I) \ind_{\{|t-s|\le \delta\}} \,{\rm d}t \,{\rm d}s\\
  \leq 2\delta n^{-2/3-\beta^3/4} + 4\beta\delta(\log^{1/3}n)n^{-2/3-\beta^3/8} +2\delta^2 n^{-1/3-\beta^3/8}\\
  \le 5\beta\delta(\log^{1/3}n)n^{-2/3-\beta^3/8} +2\delta^2 n^{-1/3-\beta^3/8}
\end{multline}

To estimate the integral when $|t-s|> \delta$, we begin with \eqref{eq:int_to_fourier}, which says that
\begin{multline*}
\int_0^1\int_0^1 \P(|\mathcal C_1(s)|\in I; |\mathcal C_2(t)|\in I)\ind_{\{|t-s|>\delta\}} \,{\rm d}t\, {\rm d}s \\
\le \E[Z_1]^2 + 2\sum_{|S|>0} \frac{e^{-\delta |S|}}{|S|} \hat f_1(S) \hat f_2(S).
\end{multline*}
We now apply Lemmas \ref{lemma:U1U2intersect} and \ref{lemma:U1U2empty}, which tell us respectively that for large $n$,
\[\sum_{S:S\cap(\mathcal U_1\cup\mathcal U_2)\neq\emptyset} \hat f_1(S)\hat f_2(S) \le C(\beta^2+\beta^{-1/2}) n^{-1-\beta^3/8}\log n\]
and
\[\sum_{\substack{S: |S|>0;\\ S\cap (\mathcal U_1\cup\mathcal U_2)=\emptyset}} \frac{e^{-\delta|S|}}{|S|} \hat f_1(S)\hat f_2(S) \leq C\delta^{-1}\E[f_1]\beta^{7/2} n^{-2/3} \log^{7/6} n .\] 
for some finite constant $C$.
Combining these three equations and noting that (by Proposition \ref{prop:pittel}(a) with $\lambda=0$) $\E[f_1]\le n^{-1/3-\beta^3/8}$, we get
\begin{multline}\label{secondint}
  \int_0^1\int_0^1 \P(|\mathcal C_1(s)|\in I; |\mathcal C_2(t)|\in I)\ind_{\{|t-s|>\delta\}}\, {\rm d}t\, {\rm d}s \\
  \le \E[Z_1]^2 + 2C(\beta^2+\beta^{-1/2}) n^{-1-\beta^3/8}\log n + 2C\delta^{-1}\beta^{7/2} n^{-1-\beta^3/8}\log^{7/6}n.
\end{multline}

Combining \eqref{firstint} with \eqref{secondint} and plugging back into \eqref{eq:Z_1Z_2_by_integral}, we get
\begin{multline*}
\E[Z_1Z_2] \le \E[Z_1]^2 + 5\beta\delta(\log^{1/3}n)n^{-2/3-\beta^3/8} +2\delta^2 n^{-1/3-\beta^3/8}\\
 + 2C(\beta^2+\beta^{-1/2}) n^{-1-\beta^3/8}\log n + 2C\delta^{-1}\beta^{7/2} n^{-1-\beta^3/8}\log^{7/6}n.
\end{multline*}
Choosing $\delta=n^{-2/9}$, the biggest term above when $n$ is large is the last one. Thus in this case there exists a finite constant $C'$ depending on $\beta$ such that 
\[\E[Z_1Z_2] \le \E[Z_1]^2 + C' n^{-7/9-\beta^3/8}\log^{7/6} n.\]
By Proposition \ref{prop:pittel}(a) (with $\lambda=0$) we know that $\E[Z_1] \ge c n^{-\beta^3/8-1/3}\log^{-1/6}n$ for some constant $c>0$, so we get
\[\frac{\E[Z_1Z_2]}{\E[Z_1]^2} \le 1 + c' n^{-1/9+\beta^3/8}\log^{3/2} n\]
for some finite constant $c'$ (depending on $\beta$). 
For $\beta<2/3^{2/3}$, the above quantity tends to $1$ as $n\to\infty$, giving
\[\liminf_{n\to\infty} \frac{\E[Z_1]^2}{\E[Z_1Z_2]} \ge 1.\]
Therefore by \eqref{eq:liminf_P}, for any $\beta<2/3^{2/3}$,
\[\liminf_{n\to\infty}\P \Big( \sup_{t\in [0,1]}|L_n(t)|> \beta n^{2/3}\log^{1/3} n \Big) = 1.\]
We have shown that exceptional times exist with high probability for any $\beta<2/3^{2/3}$, and to complete the proof of Theorem \ref{thm:noise} it remains to show that with high probability there are no such times for any $\beta \ge 2/3^{1/3}$.

\section{No exceptional times when \texorpdfstring{$\beta\ge 2/3^{1/3}$}{beta is large}}\label{sec:large_beta}

Fix $\beta>0$.
For $i \in \{0,\dots,\lfloor n^{1/3}\rfloor\}$, consider the event
\[
  \mathcal E_i:=\{\exists t\in[i n^{-1/3},(i+1)n^{1/3}) : |L_n(t)|>\beta n^{2/3}\log^{1/3}n\}.
  \]
  The probability that an edge $e$ is turned on at any time in $[i n^{-1/3}, (i+1)n^{1/3})$ is at most $1/n + (1-e^{-n^{-1/3}})/n \le (1+n^{-1/3})/n$. Therefore for each $i$,
  \[\P(\mathcal E_i) \le \mathbf P_{n,n^{-1}+n^{-4/3}}(|L_n|>\beta n^{2/3}\log^{1/3}n)\]
  where we recall that $\mathbf P_{n,p}$ is the law of an ER$(n,p)$. Applying Proposition \ref{prop:pittel}(b) with $\lambda=1$, we get that for large $n$,
  \[\P(\mathcal E_i) \le \beta^{-3/2} n^{-\beta^3/8} e^{\frac{1}{2}\beta^2\log^{2/3}n}\log^{-1/2} n,\]
  so by a union bound,
  \[\P(\exists t\in[0,1] : |L_n(t)|>\beta n^{2/3}\log^{1/3}n) \hspace{-0.5mm}\le\hspace{-0.5mm} \beta^{-3/2} n^{1/3-\beta^3/8}e^{\frac{1}{2}\beta^2\log^{2/3}n}\log^{-\frac{1}{2}} n.\]
  This tends to zero as $n\to\infty$ if $\beta\ge 2/3^{1/3}$, which shows that with high probability there are no exceptional times in this regime. This completes the proof of Theorem \ref{thm:noise}.

  \section{Proving noise sensitivity}\label{sec:ns}
  
  In this section we prove Proposition \ref{prop:ns}.
  Throughout, let $\{\eps_n\}_{n\in \N}$ be a sequence such that $\lim_{n\to\infty} n^{1/6}\eps_n=\infty$, fix $a\in(0,\infty)$ and let $F_n=\ind_{\{|L_n|\ge an^{2/3}\}}$.
  We will show that $F_n$ is quantitatively noise sensitive with scaling $\eps_n$.
  Our path will be similar to (but in some ways simpler than) the proof that exceptional times exist for small $\beta$. There is one complication: when $A_n\to\infty$, the probability that there is a component of size larger than $A_n n^{2/3}$ is approximately the expected number of vertices in such components divided by $A_n n^{2/3}$; but this is not true for $A_n=a$ fixed.
  To get around this small problem we will use the following lemma which is a consequence of the FKG inequality.
  We use the notation of Section \ref{sec:fourier}. We recall that a function $f:\Omega\to\R$ is \emph{increasing} if turning bits on can only increase the value of $f$.
  
  \begin{lemma}\label{lemma:FKGcon}
  Suppose that $f,g:\Omega\to\R$ are functions such that $\E[f^2]<\infty$, $\E[g^2]<\infty$, and both $f$ and $g-f$ are increasing. Then for any $\eps\in[0,1]$,
  \[\E[g(\omega)g(\omega_\eps)] - \E[g(\omega)]^2 \ge \E[f(\omega)f(\omega_\eps)]-\E[f(\omega)]^2.\]
  \end{lemma}
  
  \begin{proof}
  Let $h=g-f$. By applying Lemma \ref{lemma:noise_fourier} to $h$,
  \[\E[h(\omega)h(\omega_\eps)] \ge \E[h(\omega)]^2.\]
  Expanding in terms of $f$ and $g$, and rearranging, we get
  \begin{align*}
  &\E[g(\omega)g(\omega_\eps)] - \E[g(\omega)]^2 - \E[f(\omega)f(\omega_\eps)] + \E[f(\omega)]^2\\
  &\hspace{10mm}\ge \E[f(\omega)g(\omega_\eps)] + \E[f(\omega_\eps)g(\omega)]\\
  &\hspace{20mm} - 2\E[f(\omega)]\E[g(\omega)] - 2\E[f(\omega)f(\omega_\eps)] + 2\E[f(\omega)]^2\\
  &\hspace{10mm}= 2\big(\E[f(\omega)(g(\omega_\eps)-f(\omega_\eps))] - \E[f(\omega)]\E[g(\omega)-f(\omega)]\big)\\
  &\hspace{10mm}= 2\big(\E[f(\omega)h(\omega_\eps)] - \E[f(\omega)]\E[h(\omega_\eps)]\big).
  \end{align*}
  Now applying the FKG inequality to the two increasing random variables $(\omega,\omega_\eps)\mapsto f(\omega)$ and $(\omega,\omega_\eps)\mapsto h(\omega_\eps)$ shows that the last line is non-negative, and the result follows.
  \end{proof}
  
  We now follow the same strategy as in Section \ref{sec:largets}. We also use much of the same notation, just with a different value of $N$. Recall that for a vertex $v\in\{1,\ldots,n\}$,
  \[f_v = \ind_{\{|\C_v|\ge N\}}\]
  and $\mathcal U_v$ is the set of edges with an endpoint at $v$. (Of course these objects also depend on $n$, but we omit this from the notation.) Lemma \ref{lemma:FKGcon} will allow us to relate the noise sensitivity of $F_n$ to quantities involving the Fourier coefficients of $f_1$ and $f_2$, so we turn our attention to bounding those.

  Our first lemma is the equivalent of Lemma \ref{lemma:U1U2intersect}, using pivotality estimates to bound the Fourier coefficients of $f_1$ and $f_2$ on sets that intersect $\mathcal U_1\cup \mathcal U_2$.

  \begin{lemma}\label{lemma:U1U2intersect2}
  Let $N=\lceil an^{2/3}\rceil$. Then there exists a finite constant $C$ such that
  \[
    \sum_{S:S \cap (\mathcal U_1\cup\mathcal U_2)\neq \emptyset}\hat f_1(S)\hat f_2(S)\leq \frac{C}{n}.
  \]
\end{lemma}
\begin{proof}
  Just as in the proof of Lemma \ref{lemma:U1U2intersect}, we have
  \begin{equation}\label{eq:fhattoPbd2}
  \sum_{S:S \cap (\mathcal U_1\cup\mathcal U_2)\neq \emptyset}\hat f_1(S)\hat f_2(S) \leq \frac{1}{n}\P((1,2)\in\mathcal P_{f_1}\cap\mathcal P_{f_2}) + 2\P((1,3)\in \mathcal P_{f_1}\cap\mathcal P_{f_2}). 
  \end{equation}
  The first term on the right-hand side is at most $1/n$, so we can concentrate on the second term. Again following the argument to prove Lemma \ref{lemma:U1U2intersect}, 
  \[\P((1,3)\in\mathcal P_{f_1}\cap\mathcal P_{f_2}) \le 2\P\big( |\C_1\cup \C_3|\ge N, \, |\C_1| < N, \, \C_1\cap \C_3=\emptyset, \, 2\in \C_1\big).\]
  Applying Lemma \ref{lemma:Cuvw} we get
  \[\P((1,3)\in\mathcal P_{f_1}\cap\mathcal P_{f_2}) \le \frac{c}{n}\]
  for some finite constant $c$. Plugging this back into \eqref{eq:fhattoPbd2}, we have
  \[\sum_{S:S \cap (\mathcal U_1\cup\mathcal U_2)\neq \emptyset}\hat f_1(S)\hat f_2(S) \le \frac{C}{n}\]
  for some finite constant $C$.
\end{proof}

Next we bound the revealment of the breadth first search algorithm seen in Section \ref{sec:largets}, similarly to Lemma \ref{lemma:revealment}.

\begin{lemma}\label{lemma:revealment2}
  Let $A$ be the breadth first search described above Lemma \ref{lemma:revealment}, and let $N=\lceil an^{2/3}\rceil$. Then there exists a finite constant $C$ such that
  \[
    \mathcal R_{\mathcal U_1} \leq C n^{-2/3}.
  \]
\end{lemma}

\begin{proof}
  Just as in the proof of Lemma \ref{lemma:revealment}, for any edge $e=(v,w)$ with $v,w \neq 1$ we have
  \begin{equation}\label{eq:reveal_expectation_estimate2}
  \P(A \text{ reveals } \omega(e)) \leq \frac{2}{n}\E[|\mathcal S_{\tau-1}|]
  \end{equation}
  and also $\mathcal S_{\tau-1}\subset \C_1$ and $|\mathcal S_{\tau-1}|\le N$. Combining these facts, then applying Proposition~\ref{prop:pittel} and Lemma~\ref{lemma:branching_approx}, we get that
  \begin{align*}
    \P(A\text{ reveals } \omega(e)) &\leq \sum_{k=1}^{N}\frac{2k}{n}\P(|\mathcal C_1|=k) +\frac{2N}{n}\P(|\mathcal C_1|\ge N)\\ 
    &\leq  \sum_{k=1}^{N\vee n^{2/3}}\frac{2k}{n}\frac{c}{k^{3/2}} + \frac{2N}{n} \frac{c}{n^{1/3}}\\ 
    & \leq c' n^{-2/3}
  \end{align*}
  for some finite constants $c,c'$, as required.
\end{proof}

Lemma \ref{lemma:revealment2} allows us to give a bound on the Fourier coefficients of $f_1$ and $f_2$ on sets that do not intersect $\mathcal U_1$ or $\mathcal U_2$.

\begin{lemma}\label{lemma:U1U2empty2}
  Let $N=\lceil an^{2/3}\rceil$. There exists a finite constant $C$ such that for any $\eps \in(0,1)$,
  \[
    \sum_{\substack{|S|>0;\\S\cap (\mathcal U_1\cup\mathcal U_2)=\emptyset}}(1-\eps)^{|S|}\hat f_1(S)\hat f_2(S) \leq C\eps^{-2} \E[f_1] n^{-2/3}.
  \]
\end{lemma}
\begin{proof}
  Following the proof of Lemma \ref{lemma:U1U2empty}, we have
  \[\sum_{\substack{|S|>0;\\S\cap (\mathcal U_1\cup \mathcal U_2)=\emptyset}} (1-\eps)^{|S|} \hat f_1(S)\hat f_2(S) \leq \sum_{k=1}^{\binom{n}{2}} k(1-\eps)^{k} \mathcal R_{\mathcal U_1} \E[f_1^2].\]
  Since $f_1$ takes values in $\{0,1\}$, $\E[f_1^2]=\E[f_1]$, and by Lemma~\ref{lemma:revealment2}, $\mathcal R_{\mathcal U_1} \leq C n^{-2/3}$.
  Finally, note that
  \[\sum_{k=1}^\infty k(1-\eps)^k = \frac{1-\eps}{\eps^2} \le \eps^{-2}.\]
  Combining these three observations gives the desired result.
\end{proof}

We now have the tools to prove our noise sensitivity result.

\begin{proof}[Proof of Proposition \ref{prop:ns}]
  Recall that $F_n = \ind_{\{|L_n|\ge an^{2/3}\}}$ and suppose that $\lim_{n\to\infty}n^{1/6}\eps_n=\infty$.
  Define
\[G_n = \frac{1}{an^{2/3}} \sum_{v=1}^n \ind_{\{|\C_v|\ge an^{2/3}\}}.\]
Then $F_n\le G_n$ and both $F_n$ and $G_n-F_n$ are increasing, so by Lemma \ref{lemma:FKGcon} it suffices to show that
\[\E[G_n(\omega)G_n(\omega_{\eps_n})] - \E[G_n(\omega)]^2 \to 0.\]
We know from Lemma \ref{lemma:noise_fourier} that this quantity is non-negative, so it suffices to give an upper bound. But if we set $N=\lceil an^{2/3}\rceil$ then
\[G_n = \frac{1}{an^{2/3}}\sum_v f_v,\]
so
\begin{align}
&\E[G_n(\omega)G_n(\omega_{\eps_n})] - \E[G_n(\omega)]^2\nonumber\\
&\hspace{15mm} = \frac{1}{a^2 n^{4/3}} \sum_{u,v}\big(\E[f_u(\omega)f_v(\omega_{\eps_n})]-\E[f_u(\omega)]\E[f_v(\omega)]\big)\nonumber\\
&\hspace{15mm} = \frac{n}{a^2 n^{4/3}} \big(\E[f_1(\omega)f_1(\omega_{\eps_n})]-\E[f_1(\omega)]^2]\big) \nonumber\\
&\hspace{25mm}+ \frac{n(n-1)}{a^2 n^{4/3}}\big(\E[f_1(\omega)f_2(\omega_{\eps_n})]-\E[f_1(\omega)]\E[f_2(\omega)]\big)\nonumber\\
&\hspace{15mm} \le \frac{1}{a^2 n^{1/3}} + \frac{n^{2/3}}{a^2} \sum_{S\neq \emptyset} \hat f_1(S) \hat f_2(S) (1-\eps_n)^{|S|}\label{eq:Gnbd}
\end{align}
where we used Lemma \ref{lemma:noise_fourier} to get the last line.

By Lemma \ref{lemma:U1U2intersect2} we have
\[\sum_{S:S \cap (\mathcal U_1\cup\mathcal U_2)\neq \emptyset}\hat f_1(S)\hat f_2(S)\leq \frac{C}{n},\]
and by Lemma \ref{lemma:U1U2empty2} we have
\[\sum_{\substack{|S|>0;\\S\cap (\mathcal U_1\cup\mathcal U_2)=\emptyset}}(1-\eps_n)^{|S|}\hat f_1(S)\hat f_2(S) \leq C\eps_n^{-2} \E[f_1] n^{-2/3},\]
for some finite constant $C$. By Proposition \ref{prop:pittel} and Lemma \ref{lemma:branching_approx},
\[\E[f_1] \le cn^{-1/3}\]
for some finite constant $c$, and so putting the above estimates together we get
\[\sum_{S\neq\emptyset}(1-\eps_n)^{|S|}\hat f_1(S)\hat f_2(S) \leq \frac{C}{n} + \frac{C\cdot c}{\eps_n^2 n}.\]
Substituting this back into \eqref{eq:Gnbd} gives
\[\E[G_n(\omega)G_n(\omega_{\eps_n})] - \E[G_n(\omega)]^2 \le \frac{1}{a^2 n^{1/3}} + \frac{n^{2/3}}{a^2} \Big(\frac{C}{n} + \frac{C\cdot c}{\eps_n^2 n}\Big)\]
which tends to $0$ since $n^{1/3}\eps_n^2 \to \infty$.
\end{proof}

\section*{Acknowledgements}

  MR and B\c{S} are grateful for support from EPSRC grants EP/K007440/1, EP/H023348/1 and EP/L002442/1; and also thank an anonymous referee for some very useful comments and questions. MR would like to thank Jeffrey Steif for helpful conversations, and both Christophe Garban and Jeffrey Steif for their excellent lecture courses at the Clay Mathematics Summer School in 2010 and the Saint-Flour Summer School in Probability in 2012.

\end{document}